\documentclass[oneside,english,reqno]{amsart}
\usepackage[T1]{fontenc}
\usepackage[latin9]{inputenc}
\usepackage{geometry}
\usepackage{babel}
\usepackage{refstyle}
\usepackage{float}
\usepackage{mathrsfs}
\usepackage{mathtools}
\usepackage{enumitem}
\usepackage{amstext}
\usepackage{amsthm}
\usepackage{amssymb}
\usepackage{graphicx}
\usepackage[all]{xy}
\PassOptionsToPackage{normalem}{ulem}
\usepackage{ulem}
\usepackage[unicode=true,pdfusetitle,
 bookmarks=true,bookmarksnumbered=false,bookmarksopen=false,
 breaklinks=false,pdfborder={0 0 1},backref=false,colorlinks=false]
 {hyperref}
\hypersetup{
 colorlinks=true,citecolor=blue,linkcolor=blue,linktocpage=true}

\makeatletter

\AtBeginDocument{\providecommand\exaref[1]{\ref{exa:#1}}}
\AtBeginDocument{\providecommand\defref[1]{\ref{def:#1}}}
\AtBeginDocument{\providecommand\secref[1]{\ref{sec:#1}}}
\AtBeginDocument{\providecommand\lemref[1]{\ref{lem:#1}}}
\AtBeginDocument{\providecommand\subsecref[1]{\ref{subsec:#1}}}
\AtBeginDocument{\providecommand\corref[1]{\ref{cor:#1}}}
\AtBeginDocument{\providecommand\thmref[1]{\ref{thm:#1}}}
\AtBeginDocument{\providecommand\figref[1]{\ref{fig:#1}}}
\AtBeginDocument{\providecommand\remref[1]{\ref{rem:#1}}}
\RS@ifundefined{subsecref}
  {\newref{subsec}{name = \RSsectxt}}
  {}
\RS@ifundefined{thmref}
  {\def\RSthmtxt{theorem~}\newref{thm}{name = \RSthmtxt}}
  {}
\RS@ifundefined{lemref}
  {\def\RSlemtxt{lemma~}\newref{lem}{name = \RSlemtxt}}
  {}

\numberwithin{equation}{section}
\numberwithin{figure}{section}
\numberwithin{table}{section}
\theoremstyle{plain}
\newtheorem{thm}{\protect\theoremname}[section]
  \theoremstyle{remark}
  \newtheorem{rem}[thm]{\protect\remarkname}
  \theoremstyle{definition}
  \newtheorem{defn}[thm]{\protect\definitionname}
  \theoremstyle{plain}
  \newtheorem{lem}[thm]{\protect\lemmaname}
  \theoremstyle{plain}
  \newtheorem{cor}[thm]{\protect\corollaryname}
  \theoremstyle{definition}
  \newtheorem{example}[thm]{\protect\examplename}
  \theoremstyle{remark}
  \newtheorem*{rem*}{\protect\remarkname}
  \theoremstyle{plain}
  \newtheorem*{question*}{\protect\questionname}
  \theoremstyle{remark}
  \newtheorem*{acknowledgement*}{\protect\acknowledgementname}

\@ifundefined{date}{}{\date{}}
\allowdisplaybreaks
\usepackage{needspace}

\newref{lem}{refcmd={Lemma \ref{#1}}}
\newref{thm}{refcmd={Theorem \ref{#1}}}
\newref{cor}{refcmd={Corollary \ref{#1}}}
\newref{sec}{refcmd={Section \ref{#1}}}
\newref{sub}{refcmd={Section \ref{#1}}}
\newref{subsec}{refcmd={Section \ref{#1}}}
\newref{chap}{refcmd={Chapter \ref{#1}}}
\newref{prop}{refcmd={Proposition \ref{#1}}}
\newref{exa}{refcmd={Example \ref{#1}}}
\newref{tab}{refcmd={Table \ref{#1}}}
\newref{rem}{refcmd={Remark \ref{#1}}}
\newref{def}{refcmd={Definition \ref{#1}}}
\newref{fig}{refcmd={Figure \ref{#1}}}

\setlist[enumerate]{itemsep=5pt,topsep=3pt}
\setlist[enumerate,1]{label=(\roman*),ref=\roman*}
\setlist[enumerate,2]{label=(\alph*),ref=\theenumi \alph*}

\AtBeginDocument{
  
}

\makeatother

  \providecommand{\acknowledgementname}{Acknowledgement}
  \providecommand{\corollaryname}{Corollary}
  \providecommand{\definitionname}{Definition}
  \providecommand{\examplename}{Example}
  \providecommand{\lemmaname}{Lemma}
  \providecommand{\questionname}{Question}
  \providecommand{\remarkname}{Remark}
\providecommand{\theoremname}{Theorem}

\begin{document}

\title{Reflection positivity and spectral theory}

\author{Palle Jorgensen and Feng Tian}

\address{(Palle E.T. Jorgensen) Department of Mathematics, The University
of Iowa, Iowa City, IA 52242-1419, U.S.A. }

\email{palle-jorgensen@uiowa.edu}

\urladdr{http://www.math.uiowa.edu/\textasciitilde{}jorgen/}

\address{(Feng Tian) Department of Mathematics, Hampton University, Hampton,
VA 23668, U.S.A.}

\email{feng.tian@hamptonu.edu}
\begin{abstract}
We consider reflection-positivity (Osterwalder-Schrader positivity,
O.S.-p.) as it is used in the study of renormalization questions in
physics. In concrete cases, this refers to specific Hilbert spaces
that arise before and after the reflection. Our focus is a comparative
study of the associated spectral theory, now referring to the canonical
operators in these two Hilbert spaces. Indeed, the inner product which
produces the respective Hilbert spaces of quantum states changes,
and comparisons are subtle.

We analyze in detail a number of geometric and spectral theoretic
properties connected with axiomatic reflection positivity, as well
as their probabilistic counterparts; especially the role of the Markov
property. This view also suggests two new theorems, which we prove.
In rough outline: It is possible to express OS-positivity purely in
terms of a triple of projections in a fixed Hilbert space, and a reflection
operator. For such three projections, there is a related property,
often referred to as the Markov property; and it is well known that
the latter implies the former; i.e., when the reflection is given,
then the Markov property implies  O.S.-p., but not conversely. In
this paper we shall prove two theorems which flesh out a much more
precise relationship between the two. We show that for every OS-positive
system $\left(E_{+},\theta\right)$, the operator $E_{+}\theta E_{+}$
has a canonical and universal factorization. 

Our second focus is a structure theory for all admissible reflections.
Our theorems here are motivated by Phillips' theory of dissipative
extensions of unbounded operators. The word ``Markov'' traditionally
makes reference to a random walk process where the Markov property
in turn refers to past and future: Expectation of the future, conditioned
by the past. By contrast, our present initial definitions only make
reference to three prescribed projection operators, and associated
reflections. Initially, there is not even mention of an underlying
probability space. This in fact only comes later.
\end{abstract}



\maketitle
\pagestyle{myheadings}
\markright{}

\tableofcontents{}

\section{Introduction}

The notion \textquotedblleft \emph{reflection-positivity}\textquotedblright{}
came up first in a renormalization question in physics: \textquotedblleft How
to realize observables in relativistic quantum field theory (RQFT)?\textquotedblright{}
This is part of the bigger picture of quantum field theory (QFT);
and it is based on a certain analytic continuation (or reflection)
of the Wightman distributions (from the Wightman axioms). In this
analytic continuation, Osterwalder-Schrader (OS) axioms induce Euclidean
random fields; and Euclidean covariance. (See, e.g., \cite{MR0329492,MR0376002,MR545025,MR887102,MR1895530,JP13,MR3614177,MR3623255}.)
For the unitary representations of the respective symmetry groups,
we therefore change these groups as well: OS-reflection applied to
the Poincar\'e group of relativistic fields yields the Euclidean
group as its reflection. The starting point of the OS-approach to
QFT is a certain positivity condition called \textquotedblleft reflection
positivity.\textquotedblright{} 

Now, when it is carried out in concrete cases, the initial function
spaces change; but, more importantly, the inner product which produces
the respective Hilbert spaces of quantum states changes as well. What
is especially intriguing is that, before reflection we may have a
Hilbert space of functions, but after the time-reflection is turned
on, then, in the new inner product, the corresponding completion,
magically becomes a \emph{Hilbert space of distributions}.

The motivating example here is derived from a certain version of the
Segal\textendash Bargmann transform (see \exaref{rp}). For more detail
on the background and the applications, we refer to two previous joint
papers \cite{MR1641554} and \cite{MR1767902}, as well as \cite{MR0496171,MR0518418,MR659539,Jor86,MR874059,MR1294671,MR1770752,MR2337697,jor2017non}. 

Our present purpose is to analyze in detail a number of geometric
properties connected with the axioms of reflection positivity, as
well as their probabilistic counterparts; especially the role of the
Markov property. This view also suggests two new theorems, to follow
in the rest of the paper.

In rough outline: It is possible to express Osterwalder-Schrader positivity
(O.S.-p.) purely in terms of a triple of projections in a fixed Hilbert
space, and a reflection operator. For such three projections, there
is a related property, often referred to as the Markov property. It
is well known that the latter implies the former; i.e., when the reflection
is given, then the Markov property implies O.S.-p., but not conversely.

In this paper we shall prove two theorems which flesh out a much more
precise relationship between the two. The word \textquotedblleft Markov\textquotedblright{}
traditionally makes reference to a random walk process where the Markov
property in turn refers to past and future: Expectation of the future,
conditioned by the past (details below). By contrast, our present
initial definitions only make reference to three prescribed projection
operators. Initially, there is not even mention of an underlying probability
space. All this comes later. Now if we are in the context of a random
walk process, then such a process may or may not have the Markov property;
which is now instead defined relative to notions of past, present,
and future, and the associated conditional expectations.

While our discussion of the Markov property is couched here in an
axiomatic framework; and is motivated by our particular aims, we stress
that Markov properties, Markov processes, and Markov fields form an
active and very diverse area. While there are links from those directions
to our present results, the connections are not always direct. For
the readers benefit we have included the following citations \cite{MR0100796,MR0343815,MR0343816,MR0436832,MR3529898,MR3624432,MR3614559}
on Markov/random fields.

In order to make our paper accessible to non-specialists, we have
chosen to begin by recalling the fundamentals in the subject. This
choice in turn helps us outline the general framework in the form
we need it for what is to follow.

\section{\label{sec:GR}The geometry of reflections and positivity}

Let $\mathscr{H}$ be a given Hilbert space, and let $U,\theta:\mathscr{H}\rightarrow\mathscr{H}$
be two unitary operators, such that: 
\begin{gather}
\theta^{2}=I_{\mathscr{H}},\;\theta^{*}=\theta,\;\text{and}\label{eq:rp1}\\
\theta U\theta=U^{*}.\label{eq:rp2}
\end{gather}
Note that (\ref{eq:rp1}) states that $\theta$ has spectrum equal
to the two point set $\left\{ \pm1\right\} $. We think of (\ref{eq:rp2})
as a reflection symmetry for the given operator $U$. In this case,
(\ref{eq:rp2}) states that $U$ is unitarily equivalent to its adjoint
$U^{*}$, and so $U$ and its adjoint $U^{\ast}$ have the same spectrum,
but, except for trivial cases, $U$ is not selfadjoint. 

We further assume that there exists a closed subspace $\mathscr{H}_{+}\subset\mathscr{H}$
s.t. 
\begin{align}
U\mathscr{H}_{+} & \subset\mathscr{H}_{+},\;\text{and}\label{eq:rp3}
\end{align}
\begin{equation}
\left\langle h_{+},\theta h_{+}\right\rangle \geq0,\;\forall h_{+}\in\mathscr{H}_{+}.\label{eq:rp4}
\end{equation}
If $E_{+}$ is the projection onto $\mathscr{H}_{+}$, then (\ref{eq:rp4})
is equivalent to 
\begin{equation}
E_{+}\theta E_{+}\geq0\label{eq:rp5}
\end{equation}
with respect to the usual ordering of operators (see \defref{OS}). 
\begin{rem}
In our discussion of (\ref{eq:rp2})-(\ref{eq:rp3}), we state things
in the simple case of just a single unitary operator $U$, but our
conclusions will apply \emph{mutatis mutandis} also to the case when
\emph{U} is instead a strongly continuous unitary representation of
a suitable non-commutative Lie group $G$ (see \secref{MP} and the
papers cited there). In the Lie group case, there is a distinguished
one-parameter subgroup of $G$ corresponding to a choice of time-direction.
Hence the corresponding restriction will be a unitary one-parameter
group, and the forward direction will be the positive half-line $\mathbb{R}_{+}$,
viewed as a sub-semigroup. If $G$ is a Lie group, we shall also be
concerned with sub-semigroups. Condition (\ref{eq:rp3}) will refer
to invariance of $\mathscr{H}_{+}$ under this sub-semigroup. In all
these cases, we shall simply refer to $U$ with regards to (\ref{eq:rp2})-(\ref{eq:rp3}),
even if it is not a single unitary operator. In case of a single unitary
operator $U$, of course by iteration we will automatically have a
representation of the group $\mathbb{Z}$ of integers, and in this
case the sub-semigroup will be understood to be $\mathbb{N}_{0}$.
\end{rem}

\noindent \textbf{Note on terminology.} Given a fixed Hilbert space
$\mathscr{H}$, we shall make use of the following identification
between projections $P$ in $\mathscr{H}$, on the one hand, and the
corresponding closed subspaces $P\mathscr{H}\subset\mathscr{H}$ on
the other. By projection $P$, we mean an operator $P$ in $\mathscr{H}$
satisfying $P^{2}=P=P^{*}$. Conversely, if $\mathscr{L}\subset\mathscr{H}$
is a fixed closed subspace, then by general theory, we know that there
is then a unique projection, say $Q$, such that $Q\mathscr{H}=\mathscr{L}=\left\{ h\in\mathscr{H}\mathrel{;}Qh=h\right\} $. 

In some of our discussions below, there will be more than one Hilbert
space, say $\mathscr{H}$ and $\mathscr{K}$; and they may arise inside
calculations. In those cases, it will be convenient to mark the inner
products and norms with subscripts, $\left\langle \cdot,\cdot\right\rangle _{\mathscr{K}}$,
$\left\Vert \cdot\right\Vert _{\mathscr{K}}$ etc. 

In the discussion of reflection positivity, there will typically be
three projections $E_{0}$, $E_{\pm}$ at the outset, and the corresponding
closed subspaces will be denoted, $\mathscr{H}_{0}:=E_{0}\mathscr{H}$,
$\mathscr{H}_{\pm}:=E_{\pm}\mathscr{H}$. 

We shall denote such a system of projections $\left(E_{\pm},E_{0}\right)$
by $\varepsilon$. If a reflection $\theta$ (see (\ref{eq:rp1}))
maps $\mathscr{H}_{+}$ to $\mathscr{H}_{-}$ (plus minus parity),
we say that $\theta\in\mathscr{R}\left(\varepsilon\right)$. If also
(\ref{eq:rp2}) and (\ref{eq:rp3}) hold, we shall say that $\theta\in\mathscr{R}\left(\varepsilon,U\right)$.
(See \secref{MO} and \defref{E}.)

\subsection{Definitions and Lemmas}

In our study of reflections, and reflection positivity, we shall need
a number of fundamental concepts from the theory of operators in Hilbert
space. While they are in the literature, they are not collected in
a single reference. For readers not in operator theory, we include
below those basic facts in the form they will be needed inside the
paper. A new feature is the notion of signed quadratic forms and subspaces
which are positive with respect to such a given signed quadratic form;
see \lemref{QA}.
\begin{defn}
When $U$, $\theta$, and $E_{+}$ satisfy these conditions, i.e.,
(\ref{eq:rp1})-(\ref{eq:rp5}), we then say that \emph{Osterwalder-Schrader
reflection positivity} holds, abbreviated O.S.P. 
\end{defn}

Below we discuss the standard ordering of projections. What will be
important is that this ordering may be stated in terms of anyone of
six equivalent properties. Each one will be relevant for the applications
to follow; to geometry, to spectral theory, and to analysis of conditional
expectations. For the latter, see e.g., \exaref{MP}, and \subsecref{PS}.
\begin{defn}[Order on projections]
\label{def:OP} ~
\begin{enumerate}
\item A projection in a Hilbert space $\mathscr{H}$ is an operator $P$
satisfying $P=P^{2}=P^{*}$.
\item If $E$ and $P$ are two projections, we say that $E\leq P$ iff (Def.)
one of the following equivalent conditions holds:
\begin{enumerate}
\item $E\mathscr{H}\subseteq P\mathscr{H}$; 
\item $\left\Vert Eh\right\Vert \leq\left\Vert Ph\right\Vert $, $\forall h\in\mathscr{H}$; 
\item $\left\langle h,Eh\right\rangle \leq\left\langle h,Ph\right\rangle $,
$\forall h\in\mathscr{H}$;
\item $PE=E$; 
\item $EP=E$; 
\item for vectors $h\in\mathscr{H}$, the following implication holds: $Eh=h$
$\Longrightarrow$ $Ph=h$.
\end{enumerate}
\end{enumerate}
\end{defn}

\begin{proof}
This is standard operator theory, and can be found in books. See e.g.
\cite{jor2017non}. 
\end{proof}
We shall need this ordering in an analysis of system (\ref{eq:rp1})-(\ref{eq:rp5}).
From the conditions $\theta^{*}=\theta$, $\theta^{2}=I_{\mathscr{H}}$
(reflection) we conclude that $\theta=2P-I_{\mathscr{H}}$ where $P$
is the projection onto $\left\{ h\in\mathscr{H}\mid\theta h=h\right\} $. 

\begin{lem}
\label{lem:EP0}Let $\theta$ be a reflection, and let $P$ be the
projection such that $\theta=2P-I_{\mathscr{H}}$, and let $E_{0}$
be a projection; then TFAE:
\begin{enumerate}
\item $\theta E_{0}=E_{0}$;
\item $E_{0}\leq P$, i.e., $E_{0}h=h$ $\Longrightarrow$ $\theta h=h$.
\end{enumerate}
\end{lem}

\begin{proof}
We have the following equivalences:
\[
\theta E_{0}=E_{0}\Longleftrightarrow\left(2P-I_{\mathscr{H}}\right)E_{0}=E_{0}\Longleftrightarrow PE_{0}=E_{0},
\]
and the result now follows from the equivalent statements in \defref{OP}.
\end{proof}
\begin{defn}
\label{def:OS}Fix a Hilbert space $\mathscr{H}$, and let $A$ and
$B$ be two selfadjoint operators in $\mathscr{H}$. We say that $A\leq B$
iff (Def.) $\left\langle h,Ah\right\rangle \leq\left\langle h,Bh\right\rangle $,
for $\forall h\in\mathscr{H}$. 

Note that in case $A$ and $B$ are projections, this order relation
agrees with that in \defref{OP}. Also $A\geq0$, i.e., $\left\langle h,Ah\right\rangle \geq0$,
$\forall h\in\mathscr{H}$, states that the spectrum of $A$ is a
closed subset of $[0,\infty)$. 
\end{defn}

\begin{defn}
\label{def:QA}Let $\mathscr{H}$ be a Hilbert space and let $\mathscr{L}_{\pm}$
be two subspaces. Equip $\mathscr{L}_{+}\times\mathscr{L}_{-}$ with
the following \emph{signed quadratic form}, 
\begin{equation}
\left\langle x,y\right\rangle _{sig}:=\left\langle k_{+},l_{+}\right\rangle _{\mathscr{H}}-\left\langle k_{-},l_{-}\right\rangle _{\mathscr{H}},\label{eq:qa1}
\end{equation}
for all $x=\left(k_{+},k_{-}\right)$, $y=\left(l_{+},l_{-}\right)$
in $\mathscr{L}_{+}\times\mathscr{L}_{-}$. 

A subspace $\mathscr{P}\subset\mathscr{L}_{+}\times\mathscr{L}_{-}$
is said to be \emph{positive} iff for all $x=\left(k_{+},k_{-}\right)\in\mathscr{P}$,
we have 
\begin{equation}
\left\langle x,x\right\rangle _{sig}=\left\Vert k_{+}\right\Vert _{\mathscr{H}}^{2}-\left\Vert k_{-}\right\Vert _{\mathscr{H}}^{2}\geq0.\label{eq:qa2}
\end{equation}
\end{defn}

\begin{lem}
\label{lem:QA}Let $\mathscr{H}$, $\mathscr{L}_{\pm}$, and $\left\langle \cdot,\cdot\right\rangle _{sig}$
be as in \defref{QA}. Then a subspace $\mathscr{P}\subset\mathscr{L}_{+}\times\mathscr{L}_{-}$
is positive if and only if there is a contractive linear operator
$\mathscr{L}_{+}\xrightarrow{\;C\;}\mathscr{L}_{-}$ (w.r.t. the original
norm from $\mathscr{H}$) such that $\mathscr{P}$ is the graph of
$C$, and so $\mathscr{P}=\left\{ \left(k_{+},Ck_{+}\right)\mathrel{;}k_{+}\in\mathscr{L}_{+}\right\} $,
\begin{equation}
\left\langle x,x\right\rangle _{sig}=\left\Vert k_{+}\right\Vert _{\mathscr{H}}^{2}-\left\Vert Ck_{+}\right\Vert _{\mathscr{H}}^{2}.\label{qa3}
\end{equation}
\end{lem}

\begin{proof}
It is clear that the graph of a contraction is a positive subspace
in $\mathscr{L}_{+}\times\mathscr{L}_{-}$. 

Conversely, suppose $\mathscr{P}$ is a given positive subspace; then
\begin{equation}
\left\Vert k_{+}\right\Vert _{\mathscr{H}}^{2}-\left\Vert k_{-}\right\Vert _{\mathscr{H}}^{2}\geq0,\;\forall\left(k_{+},k_{-}\right)\in\mathscr{P}.\label{eq:qa4}
\end{equation}
Using (\ref{eq:qa4}), we see that if $\left(k_{+},k_{-}\right)$
and $\left(k_{+},k_{-}'\right)$ are both in $\mathscr{P}$, then
$k_{-}=k_{-}'$; and so $k_{+}\xmapsto{\;C\;}Ck_{+}=k_{-}$ defines
a unique contractive operator $\mathscr{L}_{+}\xrightarrow{\;C\;}\mathscr{L}_{-}$.
As a result, we get that $\mathscr{P}$ is then the graph of this
contraction $C$.
\end{proof}

\subsection{Reflections with given spaces $\mathscr{H}_{+}$ and $\mathscr{H}_{-}$}

The material in the previous subsection will serve to give a characterization
of families of reflections; they will be computed from positive subspaces
relative to certain signed quadratic forms; see especially \corref{EP4}.
Signed quadratic forms in an infinite dimensional setting were first
studied systematically by M. G. Krein et al \cite{MR0149287,MR0200727},
and R. S. Phillips \cite{MR0133686}.
\begin{lem}
\label{lem:EP1}Let $\mathscr{H}$, $\mathscr{H}_{+}$, $\mathscr{H}_{0}$,
and $\theta$ be as in \lemref{EP0}. Let $P$ be the projection onto
$\left\{ h\in\mathscr{H}\mathrel{;}\theta h=h\right\} $. Then 
\begin{equation}
\mathscr{H}=P\mathscr{H}\oplus\left(1-P\right)\mathscr{H}.\label{eq:ep1}
\end{equation}
The decomposition is orthogonal and therefore unique, 
\begin{equation}
h=u+v,\quad Pu=u,\quad Pv=0;\label{eq:ep2}
\end{equation}
i.e., the $\pm1$ eigenspaces for $\theta$. 

Fix a closed subspace $\mathscr{H}_{+}$. The O.S.-positivity $\left\langle h_{+},\theta h_{+}\right\rangle \geq0$,
$\forall h_{+}\in\mathscr{H}_{+}$, holds iff $\mathscr{H}_{+}$ is
contained in the graph of a contractive operator 
\begin{equation}
C:P\mathscr{H}\longrightarrow P^{\perp}\mathscr{H},\label{eq:ep3}
\end{equation}
i.e., $\mathscr{H}_{+}\subseteq\left\{ u+Cu\mathrel{;}u\in P\mathscr{H}\right\} $. 
\end{lem}

\begin{proof}
Decompose vectors $h_{+}\in\mathscr{H}_{+}$ as in (\ref{eq:ep1})-(\ref{eq:ep2}),
and assume O.S.-positivity, then 
\begin{equation}
\left\langle h_{+},\theta h_{+}\right\rangle =\left\Vert u\right\Vert ^{2}-\left\Vert v\right\Vert ^{2}\geq0;\quad h_{+}=u\oplus v\;\text{as in }\left(\ref{eq:ep2}\right).\label{eq:ep4}
\end{equation}
But then the assignment $C:u\longmapsto v$ will define a contractive
operator $C$ as stated in the lemma. Indeed, suppose $h_{+}=u\oplus v$
is as in (\ref{eq:ep4}). Since $\left\Vert u\right\Vert ^{2}-\left\Vert v\right\Vert ^{2}\geq0$;
if $u=0$, it follows that $v=0$; and so $Cu:=v$ is well defined
as a contractive operator (see \lemref{QA}).

When a contraction $C:P\mathscr{H}\rightarrow P^{\perp}\mathscr{H}$
is given, then the corresponding closed subspace $\mathscr{H}_{+}$
is $\mathscr{H}_{+}=\left\{ u+Cu\mathrel{;}u\in P\mathscr{H}\right\} $;
and the reflection $\theta=\theta_{C}$ is determined by $\theta\left(u+Cu\right):=u-Cu$,
and $\left\langle h_{+},\theta h_{+}\right\rangle =\left\Vert u\right\Vert ^{2}-\left\Vert Cu\right\Vert ^{2}\geq0$
follows. (See also \thmref{SU} below.)

Since the converse implication is clear, the lemma is proved. 
\end{proof}
\begin{cor}
\label{cor:EP1}Given $\mathscr{H}$, $\mathscr{H}_{+}$, and $\mathscr{H}_{0}$,
as stated in \lemref{EP1}. Then there is a bijection between the
admissible reflections $\theta$, on the one hand, and partially defined
contractions defined as in (\ref{eq:ep3}), on the other $C:\mathscr{H}_{+}\left(\theta\right)\longrightarrow\mathscr{H}_{-}\left(\theta\right)$
where
\begin{align*}
\mathscr{H}_{+}\left(\theta\right) & =\left\{ h\in\mathscr{H}\mathrel{;}\theta h=h\right\} ,\\
\mathscr{H}_{-}\left(\theta\right) & =\left\{ k\in\mathscr{H}\mathrel{;}\theta k=-k\right\} .
\end{align*}
\end{cor}

\begin{cor}
\label{cor:EP2}Let $\theta$ be a reflection, and let $P=proj\left\{ x\in\mathscr{H}\mathrel{;}\theta x=x\right\} $
so that $\theta=2P-I_{\mathscr{H}}$. Let $C$ be the corresponding
contraction. 

Given a projection $E_{0}$ such that $E_{0}\leq P$, then TFAE: 
\begin{enumerate}
\item \label{enu:EP2a}$E_{0}\leq E_{+}$; and
\item \label{enu:EP2b}$E_{0}\leq\ker\left(C\right)$. 
\end{enumerate}
\end{cor}

\begin{proof}
We shall identify closed subspaces in $\mathscr{H}$ with the corresponding
projections; see \defref{OP}. By \corref{EP1}, $\theta=\theta_{C}$
has the form 
\[
\theta\left(u+Cu\right)=u-Cu,\;u\in P\mathscr{H},
\]
where $C:P\mathscr{H}\rightarrow P^{\perp}\mathscr{H}$, is a uniquely
determined contraction. 

Let $x_{0}\in E_{0}$; then $x_{0}\in\mathscr{H}_{+}$ iff $\exists\left(!\right)u\in P\mathscr{H}$
such that $x_{0}=u+Cu$. So 
\[
0=\underset{\in P\mathscr{H}}{\underbrace{\left(u-x_{0}\right)}}+\underset{\in P^{\perp}\mathscr{H}}{\underbrace{Cu},}
\]
and both terms are zero; i.e., $u=x_{0}$, and $Cu=Cx_{0}=0$. The
equivalence (\ref{enu:EP2a}) $\Longleftrightarrow$ (\ref{enu:EP2b})
now follows.
\end{proof}
\begin{cor}
\label{cor:EP4}Let $\theta$ be a reflection in a Hilbert space $\mathscr{H}$,
and let $P:=proj\left\{ x\in\mathscr{H}\mathrel{;}\theta x=x\right\} $.
Let $C:P\mathscr{H}\longrightarrow P^{\perp}\mathscr{H}$ be the corresponding
contraction. Assume the subspaces $\mathscr{H}_{\pm}$ satisfy $\mathscr{H}_{+}=\left\{ x+Cx\mathrel{;}x\in P\mathscr{H}\right\} $,
and $\mathscr{H}_{-}=\theta\left(\mathscr{H}_{+}\right)=\left\{ x-Cx\mathrel{;}x\in P\mathscr{H}\right\} $.
We now have:
\begin{equation}
\mathscr{H}_{+}\cap\mathscr{H}_{-}=\ker\left(C\right)=\mathscr{H}_{+}\cap P\label{eq:ep5}
\end{equation}
where we identify subspaces with the corresponding projections.
\end{cor}

\begin{proof}
The implication ``$\supset$'' is immediate from \corref{EP2}.
Now, let $h\in\mathscr{H}_{+}\cap\mathscr{H}_{-}$. Hence, there are
vectors $x,y\in P\mathscr{H}$ such that $h=x+Cx=y-Cy$. Hence, 
\begin{equation}
\underset{\in P\mathscr{H}}{\underbrace{y-x}}=\underset{\in P^{\perp}\mathscr{H}}{\underbrace{Cx+Cy}};\label{eq:ep6}
\end{equation}
so both sides of (\ref{eq:ep6}) must be zero. We get $y=x$, and
$Cx=0$; so $h=x\in\ker\left(C\right)$ which is the desired conclusion
(\ref{eq:ep5}).
\end{proof}
\begin{rem}
\label{rem:3d}In \corref{EP4}, we assumed $\mathscr{H}_{+}=\left\{ x+Cx\mathrel{;}x\in P\mathscr{H}\right\} $;
but this is not necessarily satisfied in the general formulation (see
(\ref{eq:rp3})-(\ref{eq:rp4})). 

For example, let $\mathscr{H}=\mathbb{C}^{3}$ with the standard orthonormal
basis $\left\{ e_{j}\right\} _{j=1}^{3}$. Set 
\[
\theta:=\begin{pmatrix}1 & 0 & 0\\
0 & 1 & 0\\
0 & 0 & -1
\end{pmatrix},\;\text{and}\quad\mathscr{H}_{+}=span\left\{ e_{1}+\frac{1}{2}e_{3}\right\} .
\]
So $\mathscr{H}_{+}$ is 1-dimensional. The contraction $C$ is given
by 
\begin{align*}
C:span\left\{ e_{1}\right\}  & \longrightarrow span\left\{ e_{3}\right\} \\
Ce_{1} & =\frac{1}{2}e_{3};
\end{align*}
yields $\mathscr{H}_{+}=span\left\{ e_{1}+Ce_{1}\right\} $. Then
we have $\theta=2P-I$, where 
\[
P=\begin{pmatrix}1 & 0 & 0\\
0 & 1 & 0\\
0 & 0 & 0
\end{pmatrix},\;\text{and}\quad E_{+}\theta E_{+}\geq0,\;\text{where}
\]
$E_{+}$ denotes the projection onto $\mathscr{H}_{+}$. It is clear
that 
\[
\mathscr{H}_{+}\subsetneq\left\{ x+Cx\mathrel{;}x\in P\mathscr{H}\right\} ,\text{ proper containment,}
\]
since $\dim P=2$. 

Now, extend the contraction to $C:P\mathscr{H}\longrightarrow P^{\perp}\mathscr{H}$
via 
\[
Ce_{2}=0;
\]
then $\ker\left(C\right)=span\left\{ e_{2}\right\} $. Thus, we get
$\mathscr{H}_{\pm}=span\left\{ e_{1}\pm\frac{1}{2}e_{3}\right\} $,
but 
\[
0=\mathscr{H}_{+}\cap\mathscr{H}_{-}=\mathscr{H}_{+}\cap P\neq\ker\left(C\right)=span\left\{ e_{2}\right\} .
\]
\end{rem}

\begin{rem}
In the general configuration the two projections $E_{\pm}$ from \corref{EP4}
can be more complicated. If it is only assumed that the system $\left(E_{\pm},\theta\right)$
satisfies the O.S.-condition in (\ref{eq:rp5}), $\mathscr{H}_{\pm}:=E_{\pm}\mathscr{H}$,
then the best that can be said about $\mathscr{H}_{+}\cap\mathscr{H}_{-}$
is the following: 

Let $Q:=E_{+}\wedge E_{-}=$ the projection onto $\mathscr{H}_{+}\cap\mathscr{H}_{-}$;
then the following limit holds (in the strong operator topology):
\begin{equation}
Q=\lim_{n\rightarrow\infty}\left(E_{+}E_{-}\right)^{n}.\label{eq:ru1}
\end{equation}
This conclusion follows from a general fact in operator theory, see
e.g. \cite[sect.12]{MR0051437}, and also \cite{jor2017non}. Moreover,
the limit in (\ref{eq:ru1}) is known to be monotone (decreasing.)
\end{rem}

\subsection{Maximal Reflections}

As we saw that the specification of reflections may be stated in terms
of certain positive subspaces (\lemref{EP1}), it is natural to ask
for the corresponding notion of maximal subspaces. We address this
in the theorem to follow. The significance of maximality will be further
addressed in the subsequent section.
\begin{defn}
Let $\mathscr{H}$ be a Hilbert space and $\theta$ a reflection on
$\mathscr{H}$, see (\ref{eq:rp1}). Let $P=proj\left\{ x\in\mathscr{H}\mathrel{;}\theta x=x\right\} $,
so that $\theta=2P-I_{\mathscr{H}}$. Set 
\begin{equation}
Sub_{OS}\left(\theta\right)=\left\{ E_{+}\mathrel{;}E_{+}\text{ is a projection in \ensuremath{\mathscr{H}} s.t. \ensuremath{E_{+}\theta E_{+}\geq0}}\right\} .\label{eq:su1}
\end{equation}
\end{defn}

As usual properties for projections have equivalent formulation for
closed subspaces: In this case, we may identify elements in $Sub_{OS}\left(\theta\right)$
with closed subspaces $\mathscr{H}_{+}$ such that 
\begin{equation}
\left\langle h_{+},\theta h_{+}\right\rangle \geq0,\;\text{for \ensuremath{\forall h_{+}\in\mathscr{H}_{+}}.}\label{eq:su2}
\end{equation}
Set $\mathscr{H}_{+}:=E_{+}\mathscr{H}$. 

Now, combining the results above, we arrive at the following conclusions: 
\begin{thm}
\label{thm:SU}Let $\mathscr{H}$, $\theta$, and $P$ be as stated,
and consider the corresponding $Sub_{OS}\left(\theta\right)$ as in
(\ref{eq:su1}), or equivalently (\ref{eq:su2}). 

Then $Sub_{OS}\left(\theta\right)$ is an ordered lattice of projections,
and it has the following family of maximal elements: Let $C:P\mathscr{H}\longrightarrow P^{\perp}\mathscr{H}$
be a contractive operator, and set 
\begin{equation}
\mathscr{H}_{+}\left(P,C\right):=\left\{ x+Cx\mathrel{;}x\in P\mathscr{H}\right\} .\label{eq:su3}
\end{equation}

Then $\mathscr{H}_{+}\left(P,C\right)$ is maximal in $Sub_{OS}\left(\theta\right)$,
and every maximal element in $Sub_{OS}\left(\theta\right)$ has this
form for some contraction $C:P\mathscr{H}\rightarrow P^{\perp}\mathscr{H}$.
\end{thm}

\begin{proof}
(i) If $E_{+}\mathscr{H}=\mathscr{H}_{+}$, and $E_{+}'\mathscr{H}=\mathscr{H}_{+}'$
are in $Sub_{OS}\left(\theta\right)$, it is clear that then so is
$\left(E_{+}\wedge E_{+}'\right)\left(\mathscr{H}\right)=\mathscr{H}_{+}\cap\mathscr{H}_{+}'$. 

(ii) Fix $E_{+}\mathscr{H}=\mathscr{H}_{+}$, $E_{+}\in Sub_{OS}\left(\theta\right)$.
We have 
\begin{equation}
\mathscr{H}_{+}=P\mathscr{H}_{+}+P^{\perp}\mathscr{H}_{+},\label{eq:su4}
\end{equation}
and by the argument in the proof of \lemref{EP1}, we conclude that
there is a contractive operator $C:P\mathscr{H}_{+}\rightarrow P^{\perp}\mathscr{H}_{+}$,
and we get the representation 
\begin{equation}
\mathscr{H}_{+}=\left\{ x+Cx\mathrel{;}x\in P\mathscr{H}_{+}\right\} .\label{eq:su5}
\end{equation}

Let $\mathscr{H}_{+}^{\left(i\right)}$, $i=1,2$, be in $Sub_{OS}\left(\theta\right)$;
and suppose $\mathscr{H}_{+}^{\left(1\right)}\subseteq\mathscr{H}_{+}^{\left(2\right)}$.
Let $C_{i}$, $i=1,2$, be the corresponding contractions, i.e., $C_{i}:P\mathscr{H}_{+}^{\left(i\right)}\rightarrow P^{\perp}\mathscr{H}_{+}^{\left(i\right)}$,
then it follows from (\ref{eq:su5}) that the contraction $C_{2}$
is an extension of $C_{1}$. 

(iii) By general theory, see e.g., \cite{MR0133686}, any contraction
$C$ as in (\ref{eq:su5}) has contractive extensions $\widetilde{C}:P\mathscr{H}\rightarrow P^{\perp}\mathscr{H}$.
Setting $\mathscr{H}_{+}(P,\widetilde{C})$ as in (\ref{eq:su3}),
we conclude that $\mathscr{H}_{+}\subseteq\mathscr{H}_{+}(P,\widetilde{C})$.
Also see \cite{jor2017non}.

(iv) Converse, fix a contraction $D:P\mathscr{H}\rightarrow P^{\perp}\mathscr{H}$,
and consider $\mathscr{H}_{+}(P,D)$, as in (\ref{eq:su3}), the argument
from the proofs of \lemref{EP1} and \corref{EP1}, shows that $\mathscr{H}_{+}(P,D)$
is maximal in $Sub_{OS}\left(\theta\right)$; and, conversely, every
maximal element in $Sub_{OS}\left(\theta\right)$ has this form for
some contraction $D:P\mathscr{H}\rightarrow P^{\perp}\mathscr{H}$. 
\end{proof}
\begin{example}[1-dimensional case of $\mathscr{H}_{+}$; see (\ref{eq:su5})]
Fix $\theta=2P-I_{\mathscr{H}}$; and consider $E_{+}\in Sub_{OS}\left(\theta\right)$
with $\mathscr{H}_{+}:=E_{+}\mathscr{H}$ as spanned by $h_{+}=Pf+c\,P^{\perp}f$,
$f\in\mathscr{H}$, $c\in\mathbb{C}$, $\left\Vert f\right\Vert =1$;
and $\theta\left(h_{+}\right)=Pf-c\,P^{\perp}f$. Then 
\[
\underset{=:\alpha}{\underbrace{\left\Vert Pf\right\Vert ^{2}}}+\underset{1-\alpha}{\underbrace{\left\Vert P^{\perp}f\right\Vert ^{2}}}=1
\]
so that $\left\langle h_{+},\theta h_{+}\right\rangle =\alpha-\left|c\right|^{2}\left(1-\alpha\right)\geq0$
$\Longleftrightarrow$ $\left|c\right|^{2}\leq{\displaystyle \frac{\alpha}{1-\alpha}}.$
\end{example}

\begin{rem}
Our analysis of reflections $\theta$ and associated subspaces $\mathscr{H}_{+}$
is based on our \lemref{EP1} and \corref{EP1} where we show that
the admissible pairs $\left(\theta,\mathscr{H}_{+}\right)$ are determined
by a certain family of partially defined contractive operators. This
idea in turn is motivated by a parallel analysis of dissipative operators
with dense domain, as pioneered by R.S. Phillips, see e.g., \cite{MR0133686}.
In general, given $\theta$, there are many subspaces $\mathscr{H}_{+}$
which satisfy the O.S. positivity (\ref{eq:rp5}). In Corollaries
\ref{cor:EP2}-\ref{cor:EP4} we concentrate on a particular case
for $\mathscr{H}_{+}=\mathscr{H}_{+}\left(P\right)$ which is maximal;
see the statement of \corref{EP4}.

Our present discussion is parallel to the theory of Phillips \cite{MR0133686}
regarding dissipative extensions. Phillips' theory is also formulated
in the language of contractions. Since Phillips' theory deals with
unbounded operators with dense domain, the interesting statements
are for infinite-dimensional Hilbert spaces. Our results in \secref{MO}
below also deal with extensions, and there are many parallels between
the arguments we use there, and those of Phillips in the case of Cayley
transforms of dissipative operators.
\end{rem}

\section{New Hilbert space from reflection positivity (renormalization)}

Given a Hilbert space $\mathscr{H}$ and three closed subspaces (equivalently,
systems of projections, $E$). In this very general setting, it is
possible to give answers to the following questions: What are the
conditions on a given system $E$ which admits reflections $\theta$?
Suppose reflections exist, then fix $E$: What then is the variety
of all compatible reflections $\theta$? Characterize the maximal
reflections.

Given $E$, and an admissible reflection $(E,\theta)$, what are the
unitary operators $U$ in $\mathscr{\ensuremath{H}}$ which define
reflection symmetries with respect to $(E,\theta)$? Given $(E,\theta)$,
what is the relationship between operator theory in $\mathscr{H}_{+}$,
and that of the induced Hilbert space $\mathscr{K}$? Explore dichotomies
at the two levels.

Let $\mathscr{H}$, $\mathscr{H}_{+}$, $\theta$, and $U$ be as
above. In particular, we assume that $E_{+}\theta E_{+}\geq0$. Set
\begin{align}
\mathscr{N} & =\ker\left(E_{+}\theta E_{+}\right)=\left\{ h_{+}\in\mathscr{H}_{+}\mathrel{;}\left\langle h_{+},\theta h_{+}\right\rangle =0\right\} ,\;\text{and}\label{eq:rp6}\\
\mathscr{K} & =\left(\mathscr{H}_{+}/\mathscr{N}\right)^{\sim},\label{eq:rp6b}
\end{align}
where ``\textasciitilde{}'' in (\ref{eq:rp6b}) means Hilbert completion
with respect to the sesquilinear form: $\mathscr{H}_{+}\times\mathscr{H}_{+}\rightarrow\mathbb{C}$,
given by
\begin{equation}
\left\langle h_{+},h_{+}\right\rangle _{\mathscr{K}}:=\left\langle h_{+},\theta h_{+}\right\rangle ,\label{eq:rp7}
\end{equation}
a renormalized inner product; see (\ref{eq:rp4})-(\ref{eq:rp5}).

Set $q\left(h_{+}\right)=\text{class}\left(h_{+}\right)=h_{+}+\mathscr{N}$,
consider $q$ as a contractive operator,
\begin{equation}
\xymatrix{\mathscr{H}_{+}\ar[r]\ar@/^{1.5pc}/[rr]^{q} & \mathscr{H}_{+}/\mathscr{N}\ar[r] & \left(\mathscr{H}_{+}/\mathscr{N}\right)^{\sim}}
=\text{Hilbert completion \ensuremath{=\mathscr{K}}}.\label{eq:rp7a}
\end{equation}
\begin{rem*}
Constructing physical Hilbert spaces entail completions, often a completion
of a suitable space of functions. What can happen is that the completion
may fail to be a Hilbert space of \emph{functions}, but rather a suitable
Hilbert space of \emph{distributions}. Recall that a completion, say
$\mathscr{H}$, is defined axiomatically, and the ``real'' secret
is revealed only when the elements in $\mathscr{H}$ are identified;
see \exaref{rp} below. 
\end{rem*}

\subsection{Factorizations of $E_{+}\theta E_{+}$}

Given the basic framework of OS reflection positivity, the operator
$E_{+}\theta E_{+}$ plays a crucial role since OS positivity is defined
directly from this operator. We show that the operator $q$ from (\ref{eq:rp7a})
offers a canonical factorization of $E_{+}\theta E_{+}=q^{*}q$. But
we further show that this factorization is universal; see \corref{TB}.
\begin{thm}
\label{thm:TB}Let $\mathscr{H}$, $\theta$, $E_{+}$ be as above,
$\mathscr{H}_{+}:=E_{+}\mathscr{H}$. Then TFAE:
\begin{enumerate}
\item \label{enu:TBa}$E_{+}\theta E_{+}\geq0$, O.S.-positivity; and
\item \label{enu:TBb}there is a Hilbert space $\mathscr{L}$, and a bounded
operator $B:\mathscr{H}_{+}\rightarrow\mathscr{L}$ such that 
\begin{equation}
E_{+}\theta E_{+}=B^{*}B;\label{eq:tb1}
\end{equation}
see \figref{TB}. 
\end{enumerate}
\end{thm}

\begin{rem}
We show below that $\mathscr{H}_{+}\xrightarrow{\;q\;}\mathscr{K}$
is a \emph{universal }solution to the factorization problem (\ref{eq:tb1})
(see \corref{TB}). 
\end{rem}

\begin{proof}[Proof of \thmref{TB}]
The implication (\ref{enu:TBa})$\Longrightarrow$(\ref{enu:TBb})
is contained in \lemref{TB} below. Indeed, if (\ref{enu:TBa}) holds,
then we may take $\mathscr{L}=\mathscr{K}$, and $B=q:\mathscr{H}_{+}\rightarrow\mathscr{K}$;
see (\ref{eq:rp7a}). 

Conversely; suppose (\ref{enu:TBb}) holds (see \figref{TB}), then
it is immediate that $E_{+}\theta E_{+}=B^{*}B\geq0$, by general
theory; see \defref{OS} above.
\end{proof}
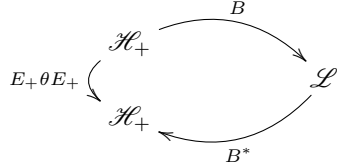
\begin{figure}[H]
\[ \xymatrix@R-2pc{\mathscr{H}_{+}\ar@/_{1.3pc}/[dd]_{E_{+}\theta E_{+}}\ar@/^{1.3pc}/[rrd]^{B}\\  &  & \mathscr{L}\ar@/^{1.3pc}/[dll]^{B^{*}}\\ \mathscr{H}_{+} } \]

\caption{\label{fig:TB}A factorization of $E_{+}\theta E_{+}$.}
\end{figure}
\begin{lem}
\label{lem:TB}Let $\mathscr{H}$, $\theta$, $E_{+}$ be as above.
We assume further that $E_{+}\theta E_{+}\geq0$, i.e., O.S.-positivity
holds. Set $\mathscr{H}_{+}=E_{+}\mathscr{H}$. Let $\mathscr{K}$
be the induced Hilbert space 
\begin{equation}
\mathscr{K}=\left(\mathscr{H}_{+}/\left\{ h_{+}\mathrel{;}\left\langle h_{+},\theta h_{+}\right\rangle =0\right\} \right)^{\sim}\label{eq:k1}
\end{equation}
as in (\ref{eq:rp7a}), and let $q:\mathscr{H}_{+}\rightarrow\mathscr{K}$
be the canonical contraction. Then the adjoint operator $q^{*}:\mathscr{K}\rightarrow\mathscr{H}_{+}$
is given by 
\begin{equation}
q^{*}\left(q\left(h_{+}\right)\right)=E_{+}\theta h_{+},\;\forall h_{+}\in\mathscr{H}_{+}.\label{eq:k2}
\end{equation}
In particular, the formula (\ref{eq:k2}) defines $q^{*}$ unambiguously. 
\end{lem}

\begin{proof}
(i) We first show that the formula (\ref{eq:k2}) defines an operator:
We must show that if 
\begin{equation}
\left\langle h_{+},\theta h_{+}\right\rangle =0,\label{eq:k3}
\end{equation}
then $E_{+}\theta h_{+}=0$. But by Schwarz, for all $l_{+}\in\mathscr{H}_{+}$,
we have 
\[
\left|\left\langle l_{+},\theta h_{+}\right\rangle \right|^{2}\leq\left\langle l_{+},\theta l_{+}\right\rangle \left\langle h_{+},\theta h_{+}\right\rangle \underset{\text{by }\left(\ref{eq:k3}\right)}{=}0
\]
and so $E_{+}\theta h_{+}=0$ as required in (\ref{eq:k2}).

(ii) Since $q^{*}$ is contractive, it is determined uniquely by its
values on a dense subspace of vectors in $\mathscr{K}$; in this case
$\left\{ q\left(h_{+}\right)\mathrel{;}h_{+}\in\mathscr{H}_{+}\right\} $. 

(iii) It remains to verify that 
\begin{equation}
\left\langle q^{*}\left(q\left(h_{+}\right)\right),l_{+}\right\rangle _{\mathscr{H}}=\left\langle E_{+}\theta h_{+},l_{+}\right\rangle _{\mathscr{H}}=\left\langle h_{+},\theta l_{+}\right\rangle _{\mathscr{H}}\left(=\left\langle q\left(h_{+}\right),q\left(l_{+}\right)\right\rangle _{\mathscr{K}}\right),\label{eq:k4}
\end{equation}
$\forall h_{+},l_{+}\in\mathscr{H}_{+}$. Details: 
\begin{align*}
\text{LHS}_{\left(\ref{eq:k4}\right)} & =\left\langle E_{+}\theta h_{+},l_{+}\right\rangle =\left\langle \theta h_{+},E_{+}l_{+}\right\rangle \\
 & =\left\langle \theta h_{+},l_{+}\right\rangle =\left\langle h_{+},\theta l_{+}\right\rangle =\text{RHS}_{\left(\ref{eq:k4}\right)}
\end{align*}
where we used the assumptions (\ref{eq:rp1}) and (\ref{eq:rp5}).
In the computation, we omitted the subscript $\mathscr{H}$ in the
inner products. 
\end{proof}
\begin{cor}
\label{cor:TB}The solution $q:\mathscr{H}_{+}\rightarrow\mathscr{K}$
to the factorization problem $E_{+}\theta E_{+}=q^{*}q$ (see (\ref{eq:tb1})),
in the O.S.-p. case, is \uline{universal} in the sense that if
$\mathscr{H}_{+}\xrightarrow{\;B\;}\mathscr{L}$ is any solution to
(\ref{eq:tb1}) in \thmref{TB}, then there is a unique isomorphism
$\mathscr{K}\xrightarrow{\;b\;}\mathscr{L}$ such that $b\,q=B$,
see \figref{TB2}; and $b^{*}b=I_{\mathscr{K}}$, so $b$ is isometric. 
\end{cor}

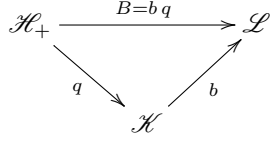
\begin{figure}
\[
\xymatrix{\mathscr{H}_{+}\ar[rr]^{B=b\,q}\ar[dr]_{q} &  & \mathscr{L}\\
 & \mathscr{K}\ar[ru]_{b}
}
\]
\caption{\label{fig:TB2}Universality of $q$.}

\end{figure}
\begin{proof}
Let $\mathscr{H}_{+}\xrightarrow{\;B\;}\mathscr{L}$ be a solution
to (\ref{eq:tb1}) in \thmref{TB}; we then define the isomorphism
$b$ (so as to complete the diagram in \figref{TB2}) as follow: 

For $h_{+}\in\mathscr{H}_{+}$, set 
\begin{equation}
b\left(q\left(h_{+}\right)\right):=B\left(h_{+}\right).\label{eq:tb2}
\end{equation}
Now this defines an operator $b:\mathscr{K}\rightarrow\mathscr{L}$,
since if $q\left(h_{+}\right)=0$, then $0=q^{*}q\left(h_{+}\right)=E_{+}\theta E_{+}=B^{*}B\left(h_{+}\right)$,
so $0=\left\langle h_{+},B^{*}Bh_{+}\right\rangle =\left\Vert Bh_{+}\right\Vert ^{2}$,
and so $Bh_{+}=0$ as required. 

Now it is immediate from (\ref{eq:tb2}), that this operator $b:\mathscr{K}\rightarrow\mathscr{L}$
has the desired properties, in particular that the universality holds;
see \figref{TB2}. 
\end{proof}
\begin{lem}
\label{lem:CP}Let $\mathscr{H}$ be a Hilbert space, and $\theta$
a reflection in $\mathscr{H}$ (see (\ref{eq:rp1})). Let $P:=proj\left\{ x\in\mathscr{H}\mathrel{;}\theta x=x\right\} $,
so $\theta=2P-I_{\mathscr{H}}$. Let $\mathscr{K}$ be the new Hilbert
space in (\ref{eq:rp7a}). Let $C:P\mathscr{H}\longrightarrow P^{\perp}\mathscr{H}$
be the contraction, such that 
\begin{equation}
\mathscr{H}_{+}=\left\{ x+Cx\mathrel{;}x\in P\mathscr{H}\right\} ,\label{eq:cp1}
\end{equation}
and $\theta\left(x+Cx\right)=x-Cx$; then for $h_{+}=x+Cx$, we have
\begin{equation}
\left\langle h_{+},\theta h_{+}\right\rangle _{\mathscr{H}}=\left\Vert h_{+}\right\Vert _{\mathscr{K}}^{2}=\left\Vert \left(I_{\mathscr{H}}-C^{*}C\right)^{\frac{1}{2}}x\right\Vert _{\mathscr{H}}^{2}.\label{eq:cp2}
\end{equation}
\end{lem}

\begin{proof}
By $\mathscr{K}$ we refer here to the completion (\ref{eq:rp7a});
see also \figref{rp}. For the LHS in (\ref{eq:cp2}), we have
\begin{align*}
\left\langle h_{+},\theta h_{+}\right\rangle  & =\left\langle x+Cx,x-Cx\right\rangle \\
 & =\left\Vert x\right\Vert ^{2}-\left\Vert Cx\right\Vert ^{2}\\
 & =\left\Vert x\right\Vert ^{2}-\left\langle x,C^{*}Cx\right\rangle \\
 & =\left\langle x,\left(I-C^{*}C\right)x\right\rangle \\
 & =\left\Vert \left(I-C^{*}C\right)^{\frac{1}{2}}x\right\Vert ^{2}=\text{RHS}_{\left(\ref{eq:cp2}\right)},
\end{align*}
where we have dropped the subscript $\mathscr{H}$ in the computation.
\end{proof}
\begin{rem}
The conclusion in \lemref{CP} states that the range $Ran\big(\left(I-C^{*}C\right)^{\frac{1}{2}}\big)$
is a realization of the induced Hilbert space $\mathscr{K}$ in (\ref{eq:rp7a}),
so 
\[
\left\Vert q\left(h_{+}\right)\right\Vert _{\mathscr{K}}=\left\Vert \left(I-C^{*}C\right)^{\frac{1}{2}}x\right\Vert _{\mathscr{H}}
\]
where $h_{+}=x+Cx$, $x\in P\mathscr{H}$. 
\end{rem}

\begin{lem}
\label{lem:UK}Let the setting be as above, see (\ref{eq:rp1})-(\ref{eq:rp3}).
Then $\widetilde{U}:\mathscr{K}\rightarrow\mathscr{K}$, given by
\begin{equation}
\widetilde{U}\left(\mbox{class}\:h_{+}\right)=\mbox{class}\left(Uh_{+}\right),\;h_{+}\in\mathscr{H}_{+}\label{eq:rp8}
\end{equation}
where class $h_{+}$ refers to the quotient in (\ref{eq:rp6}), is
\emph{selfadjoint} and \emph{contractive} (see \figref{rp}).
\end{lem}

\begin{proof}
(See \cite{MR0496171,Jor86,MR874059,MR1641554,MR1895530}.) Despite
the fact that proof details in one form or the other are in the literature,
we feel that the spectral theoretic features of the argument have
not been stressed; at least not in a form which we shall need below.

Denote the ``new'' inner product in $\mathscr{K}$ by $\left\langle \cdot,\cdot\right\rangle _{\mathscr{K}}$,
and the initial inner product in $\mathscr{H}$ by $\left\langle \cdot,\cdot\right\rangle $. 

\emph{$\widetilde{U}$ is symmetric}: Let $x,y\in\mathscr{H}_{+}$,
then 
\[
\langle x,\widetilde{U}y\rangle_{\mathscr{K}}=\left\langle x,\theta Uy\right\rangle =\left\langle x,U^{*}\theta y\right\rangle =\left\langle Ux,\theta y\right\rangle =\langle\widetilde{U}x,y\rangle_{\mathscr{K}}
\]
which is the desired conclusion.\emph{ }

\emph{$\widetilde{U}$ is contractive}: Let $x\in\mathscr{H}_{+}$,
then 
\begin{align*}
\left\Vert \widetilde{U}x\right\Vert _{\mathscr{K}}^{2} & =\left\langle Ux,\theta Ux\right\rangle =\left\langle Ux,U^{*}\theta x\right\rangle \\
 & =\left\langle U^{2}x,\theta x\right\rangle =\left\langle \widetilde{U}^{2}x,x\right\rangle _{\mathscr{K}}\\
 & \leq\left\Vert \widetilde{U}^{2}x\right\Vert _{\mathscr{K}}\cdot\left\Vert x\right\Vert _{\mathscr{K}}\qquad\left(\mbox{by Schwarz in}\;\mathscr{K}\right)\\
 & \leq\left\Vert \widetilde{U}^{4}x\right\Vert _{\mathscr{K}}^{\frac{1}{2}}\cdot\left\Vert x\right\Vert _{\mathscr{K}}^{1+\frac{1}{2}}\qquad\left(\mbox{by the first step}\right)\\
 & \leq\left\Vert \widetilde{U}^{2^{n+1}}x\right\Vert _{\mathscr{K}}^{\frac{1}{2^{n}}}\cdot\left\Vert x\right\Vert _{\mathscr{K}}^{1+\frac{1}{2}+\cdots+\frac{1}{2^{n}}}.\qquad\left(\mbox{by iteration}\right)
\end{align*}
By the spectral-radius formula, $\lim_{n\rightarrow\infty}\left\Vert \widetilde{U}^{2^{n}}x\right\Vert _{\mathscr{K}}^{\frac{1}{2^{n}}}=1$;
and we get $\left\Vert \widetilde{U}x\right\Vert _{\mathscr{K}}^{2}\leq\left\Vert x\right\Vert _{\mathscr{K}}^{2}$,
which is the desired contractivity.
\end{proof}
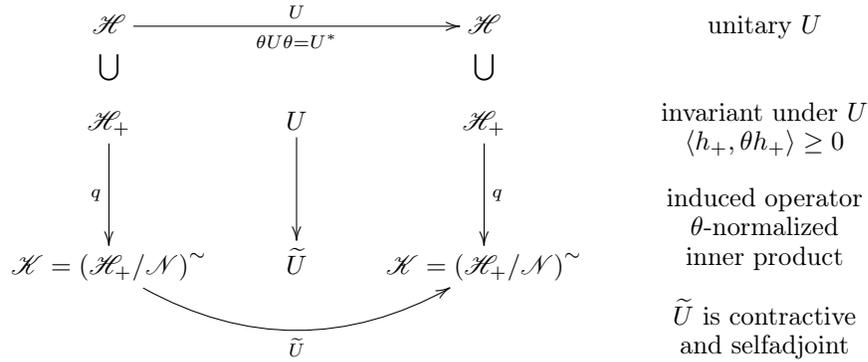
\begin{figure}
\[ \xymatrix@R-2pc{\mathscr{H}\ar[rr]_{\theta U \theta=U^{*}}^{U} &  & \mathscr{H} & \text{unitary }U\\ \bigcup &  & \bigcup\\ \mathscr{H}_{+}\ar[dd]_{q} & U\ar[dd] & \mathscr{H}_{+}\ar[dd]^{q} & \txt{invariant under \ensuremath{U} \\ \ensuremath{\left\langle h_{+},\theta h_{+}\right\rangle \geq0} }\\ \\ \mathscr{K}=\left(\mathscr{H}_{+}/\mathscr{N}\right)^{\sim}\ar@/_{2pc}/[rr]_{\widetilde{U}} & \widetilde{U} & \mathscr{K}=\left(\mathscr{H}_{+}/\mathscr{N}\right)^{\sim} & \txt{induced operator \\ \ensuremath{\theta}-normalized\\ inner product \\ \\ \ensuremath{\widetilde{U}} is contractive \\and selfadjoint} } \]

\caption{\label{fig:rp}Reflection positivity. A unitary operator $U$ transforms
into a selfadjoint contraction $\widetilde{U}$.}
\end{figure}
\begin{rem}
In the proof of \lemref{UK}, we have made an identification: 
\[
\mathscr{H}_{+}\ni x\longleftrightarrow q\left(x\right)\in\mathscr{K},
\]
see (\ref{eq:rp7a}). So the precise vectors are as follows: $\widetilde{U}q\left(x\right)=q\left(Ux\right)$,
$\left(x\in\mathscr{H}_{+}\right)$; see \figref{rp}. The proof is
in two steps: 

\uline{Step 1}. We verify the two conclusions for $\widetilde{U}$
(symmetry and contractivity) but only initially for the dense space
of vectors in $\mathscr{K}$: $\left\{ q\left(x\right)\mathrel{;}x\in\mathscr{H}_{+}\right\} $. 

\uline{Step 2}. Having the two properties verified on a dense subspace
in $\mathscr{K}$, it follows that the same conclusions will hold
also on $\mathscr{K}:=$ completion of $\left\{ q\left(x\right)\mathrel{;}x\in\mathscr{H}_{+}\right\} $.
The reason is that the two properties are preserved by passing to
limits; now limit in the $\mathscr{K}$-norm. 
\end{rem}

\begin{lem}
Let $\mathscr{H}$, $\mathscr{H}_{+}$, and $\theta$ be as above.
Set 
\begin{align*}
\mathscr{A}_{+}: & =\big\{ U\in\mathscr{H}\rightarrow\mathscr{H},\;\text{bounded operators},\\
 & \qquad U\mathscr{H}_{+}\subset\mathscr{H}_{+}\left(E_{+}UE_{+}=UE_{+}\right),\;\text{and }\theta U=U^{*}\theta\big\},
\end{align*}
then $U,V\in\mathscr{A}_{+}\Longrightarrow$ $UV\in\mathscr{A}_{+}$,
and $\left(UV\right)^{\sim}=\widetilde{U}\widetilde{V}$, where $\widetilde{U}$
is determined by 
\[
\widetilde{U}\left(q\left(h_{+}\right)\right)=q\left(Uh_{+}\right),\;\forall h_{+}\in\mathscr{H}_{+}.
\]
\end{lem}

\begin{proof}
Immediate from \lemref{UK}.
\end{proof}
\begin{lem}
\label{lem:EP}Let $\mathscr{H}$ be a fixed Hilbert space with subspaces
$\mathscr{H}_{\pm}$ and $\mathscr{H}_{0}$. Let $E_{\pm}$ and $E_{0}$
denote the respective projections. Let $\theta$ be a reflection,
i.e., $\theta^{2}=I_{\mathscr{H}}$, $\theta^{*}=\theta$. Assume
\begin{align}
E_{-}\theta E_{+} & =\theta E_{+};\nonumber \\
E_{+}\theta E_{-} & =\theta E_{-};\;\text{and}\label{eq:p1}\\
\theta E_{0} & =E_{0}.\nonumber 
\end{align}
\begin{enumerate}
\item \label{enu:EP1}Suppose $\theta:\mathscr{H}_{+}\rightarrow\mathscr{H}_{-}$
is onto. Then we have the following equivalence 
\begin{equation}
E_{+}\theta E_{+}\geq0\Longleftrightarrow E_{-}\theta E_{-}\geq0.\label{eq:p2}
\end{equation}
\item \label{enu:EP2}Suppose (\ref{enu:EP1}) holds, then we get two completions
\begin{equation}
\mathscr{K}_{\pm}:=\left(\mathscr{H}_{\pm}/\left\{ h_{\pm}\mathrel{;}\left\langle h_{\pm},\theta h_{\pm}\right\rangle =0\right\} \right)^{\sim},\label{eq:rp12}
\end{equation}
see (\ref{eq:rp7a}) above. Then $\theta$ induces two isometries
$\widetilde{\theta}:\mathscr{K}_{+}\rightarrow\mathscr{K}_{-}$, $\widetilde{\theta}:\mathscr{K}_{-}\rightarrow\mathscr{K}_{+}$, 
\item In general, the isometries from (\ref{enu:EP2}) are not onto. Indeed,
$\widetilde{\theta}:\mathscr{K}_{+}\rightarrow\mathscr{K}_{-}$ is
onto iff $\mathscr{H}_{-}\ominus\theta\mathscr{H}_{+}=0$; and $\widetilde{\theta}:\mathscr{K}_{-}\rightarrow\mathscr{K}_{+}$
is onto iff $\mathscr{H}_{+}\ominus\theta\mathscr{H}_{-}=0$. 
\end{enumerate}
\end{lem}

\begin{proof}
The key step in the proof of the lemma is (\ref{eq:p2}). Indeed we
have the following: 
\begin{align*}
E_{+}\theta E_{+} & \geq0;\\
 & \Updownarrow\\
\left\langle h_{+},\theta h_{+}\right\rangle  & \geq0,\;\forall h_{+}\in\mathscr{H}_{+};\\
 & \Updownarrow\\
\left\langle \theta h_{+},\theta^{2}h_{+}\right\rangle  & \geq0,\;\forall h_{+}\in\mathscr{H}_{+};\\
 & \Updownarrow\\
\left\langle h_{-},\theta h_{-}\right\rangle  & \geq0,\;\forall h_{-}=\theta\left(h_{+}\right)\in\mathscr{H}_{-},
\end{align*}
where we used assumption (\ref{eq:p1}) above. 

Moreover, for all $h_{+}\in\mathscr{H}_{+}$, we have:
\begin{align*}
\left\Vert \text{class}\left(\theta h_{+}\right)\right\Vert _{\mathscr{K}_{-}}^{2} & =\left\langle \theta h_{+},\theta\theta h_{+}\right\rangle \\
 & =\left\langle h_{+},\theta h_{+}\right\rangle =\left\Vert \text{class}\left(h_{+}\right)\right\Vert _{\mathscr{H}_{+}}^{2}.
\end{align*}
The remaining part of the proof is left to the reader.
\end{proof}
We now turn to a closer examination of the unitary reflection operator
$U$ from (\ref{eq:rp1})-(\ref{eq:rp3}). Given $\theta$ as in (\ref{eq:rp1}),
i.e., $\theta=\theta^{*}$, $\theta^{2}=I_{\mathscr{H}}$; we assume
that $\mathscr{H}_{\pm}$ are two closed subspaces in $\mathscr{H}$
such that $\theta\mathscr{H}_{+}\subset\mathscr{H}_{-}$; or, equivalently,
$E_{-}\theta E_{+}=\theta E_{+}$, where $E_{\pm}$ denote the respective
projection for the corresponding subspaces $\mathscr{H}_{\pm}$; i.e.,
\begin{equation}
\mathscr{H}_{\pm}=\left\{ h_{\pm}\in\mathscr{H}\mathrel{;}E_{\pm}h_{\pm}=h_{\pm}\right\} .\label{eq:re1}
\end{equation}
Finally, we shall assume that the O.S.-positivity condition $E_{+}\theta E_{+}\geq0$
holds; and so we are in a position to apply \lemref{EP1} and \corref{EP1}
above.

A given unitary operator $U$ in $\mathscr{H}$ is said to be a \emph{reflection-symmetry}
iff (Def.) 
\begin{align}
\theta U\theta & =U^{*};\;\text{and}\label{eq:re2}\\
U\mathscr{H}_{+} & \subseteq\mathscr{H}_{+}\left(\text{equivalently, }E_{+}UE_{+}=UE_{+}.\right)\label{eq:re3}
\end{align}
\begin{thm}
Let $\mathscr{H}$, $\mathscr{H}_{\pm}$, $\theta$, and $U$ be as
above, i.e., we are assuming O.S.-positivity; and further that $U$
satisfies (\ref{eq:re2})-(\ref{eq:re3}). Let $P$ be the projection
onto $\left\{ h\in\mathscr{H}\mathrel{;}\theta h=h\right\} $, i.e.,
we have $\theta=2P-I_{\mathscr{H}}$.
\begin{enumerate}
\item \label{enu:re1}Then 
\begin{equation}
PUE_{+}=PU^{*}\theta E_{+}.
\end{equation}
\item \label{enu:re2}If $C:P\mathscr{H}\longrightarrow P^{\perp}\mathscr{H}$
denotes the contraction from \lemref{EP1} and \corref{EP1}, then
there is a unique operator $U_{P}:P\mathscr{H}\longrightarrow P\mathscr{H}$
such that $U_{P}=PUP$; and, if $h_{+}=x+Cx$, $x\in P\mathscr{H}$,
then 
\begin{equation}
\left\Vert \widetilde{U}q\left(h_{+}\right)\right\Vert _{\mathscr{K}}^{2}=\left\Vert U_{P}x\right\Vert _{\mathscr{H}}^{2}-\left\Vert CU_{P}x\right\Vert _{\mathscr{H}}^{2}.\label{eq:re5}
\end{equation}
\item \label{enu:re3}In particular, since  $\widetilde{U}$ is contractive
by \lemref{UK}, we have 
\[
\left\Vert U_{P}x\right\Vert _{\mathscr{H}}^{2}-\left\Vert CU_{P}x\right\Vert _{\mathscr{H}}^{2}\leq\left\Vert x\right\Vert _{\mathscr{H}}^{2}-\left\Vert Cx\right\Vert _{\mathscr{H}}^{2},\;\forall x\in P\mathscr{H}.
\]
 
\end{enumerate}
\end{thm}

\begin{proof}
Note that (\ref{enu:re1}) is immediate from (\ref{eq:rp2}) and \corref{EP1}. 

The first half is immediate from definition of the contraction $C$
from \lemref{EP1}. For $h_{+}=x+Cx$, $x\in P\mathscr{H}$, we have
\[
\left\langle h_{+},\theta h_{+}\right\rangle _{\mathscr{H}}=\left\Vert q\left(h_{+}\right)\right\Vert _{\mathscr{K}}^{2}=\left\Vert x\right\Vert _{\mathscr{H}}^{2}-\left\Vert Cx\right\Vert _{\mathscr{H}}^{2},
\]
and 
\[
\left\Vert \widetilde{U}\left(q\left(h_{+}\right)\right)\right\Vert _{\mathscr{K}}^{2}=\left\Vert q\left(Uh_{+}\right)\right\Vert _{\mathscr{K}}^{2}=\left\Vert U_{P}x\right\Vert _{\mathscr{H}}^{2}-\left\Vert CU_{P}x\right\Vert _{\mathscr{H}}^{2};
\]
and eq. (\ref{eq:re5}) in (\ref{enu:re2}) follows. 

Now (\ref{enu:re3}) is immediate from (\ref{enu:re1})-(\ref{enu:re2})
combined with the fact that $\widetilde{U}$ is contractive in $\mathscr{K}$;
see \lemref{UK}.
\end{proof}
\begin{cor}
\label{cor:us1}Let $\mathscr{H}$, $\mathscr{H}_{\pm}$, $\mathscr{H}_{0}$,
$E_{\pm}$, $E_{0}$, $\theta$, be as in the statement of \lemref{EP}.
Let $\mathscr{K}_{\pm}$ be the corresponding induced Hilbert spaces,
see (\ref{eq:rp12}). Now set 
\begin{equation}
\mathscr{H}_{\pm}^{ex}=\text{closed span of }\left\{ h_{0}+h_{\pm}\mathrel{;}h_{0}\in\mathscr{H}_{0},\:h_{\pm}\in\mathscr{H}_{\pm}\right\} ,\label{eq:us1}
\end{equation}
and let $E_{\pm}^{ex}$ denote the corresponding projections, i.e.,
$E_{\pm}^{ex}:=E_{0}\vee E_{\pm}$. Then the following analogies of
(\ref{eq:p1}) hold: 
\begin{align}
E_{-}^{ex}\theta E_{+}^{ex} & =\theta E_{+}^{ex};\;\text{and}\label{eq:us2}\\
E_{+}^{ex}\theta E_{-}^{ex} & =\theta E_{-}^{ex}.\label{eq:us3}
\end{align}
Moreover, we have the implication
\begin{equation}
E_{+}\theta E_{+}\geq0\Longrightarrow E_{+}^{ex}\theta E_{+}^{ex}\geq0,\label{eq:us3a}
\end{equation}
if and only if 
\begin{equation}
\left|\left\langle h_{+},h_{0}\right\rangle \right|^{2}\leq\left\langle h_{+},\theta h_{+}\right\rangle \left\Vert h_{0}\right\Vert ^{2},\;\forall h_{+}\in\mathscr{H}_{+},\forall h_{0}\in\mathscr{H}_{0}.\label{eq:us3b}
\end{equation}
\end{cor}

\begin{proof}
By \lemref{EP}, it is easy to prove one of the two formula (\ref{eq:us2})-(\ref{eq:us3}).

In detail, we must show that if $h_{0}\in\mathscr{H}_{0}$, $h_{+}\in\mathscr{H}_{+}$,
then $\theta\left(h_{0}+h_{+}\right)\in\mathscr{H}_{-}^{ex}$; see
(\ref{eq:us1}). But this is clear since 
\begin{equation}
\theta\left(h_{0}+h_{+}\right)=\theta h_{0}+\theta h_{+}=h_{0}+\theta h_{+},\label{eq:us4}
\end{equation}
and $\theta h_{+}\in\mathscr{H}_{-}$ by (\ref{eq:p1}). We also used
$\theta h_{0}=h_{0}$ which is (ii) in \lemref{EP0}.

The second conclusion follows from this, since if $\left\langle h_{+},\theta h_{+}\right\rangle \geq0$,
$\forall h_{+}\in\mathscr{H}_{+}$; then 
\begin{eqnarray*}
\left\langle h_{+}+h_{0},\theta\left(h_{+}+h_{0}\right)\right\rangle  & \underset{\text{by }\left(\ref{eq:us4}\right)}{=} & \left\langle h_{+}+h_{0},\theta h_{+}+h_{0}\right\rangle \\
 & = & \left\langle h_{+},\theta h_{+}\right\rangle +\left\langle h_{+},h_{0}\right\rangle +\left\langle h_{0},\theta h_{+}\right\rangle +\left\Vert h_{0}\right\Vert ^{2}.
\end{eqnarray*}
Now use $\left\langle h_{0},\theta h_{+}\right\rangle =\left\langle \theta h_{0},h_{+}\right\rangle =\left\langle h_{0},h_{+}\right\rangle $,
and the result follows. 
\end{proof}
\begin{rem}
In the statement of \corref{us1}, we impose the technical assumption
(\ref{eq:us3b}). The following example shows that this restricting
condition (\ref{eq:us3b}) does not always hold; i.e., that \corref{us1}
cannot be strengthened.
\end{rem}

\begin{example}[Also see \remref{3d}]
 Let $\mathscr{H}=\mathbb{C}^{3}$ with standard orthonormal basis
$\left\{ e_{j}\right\} _{j=1}^{3}$. Consider the reflection
\[
\theta=\begin{pmatrix}1 & 0 & 0\\
0 & 1 & 0\\
0 & 0 & -1
\end{pmatrix},
\]
and set 
\begin{align*}
\mathscr{H}_{+} & =span\left\{ e_{1}+\frac{1}{2}e_{3}\right\} ,\\
\mathscr{H}_{-} & =span\left\{ e_{1}-\frac{1}{2}e_{3}\right\} ,\\
\mathscr{H}_{0} & =span\left\{ e_{1}\right\} .
\end{align*}
For $h_{+}:=e_{1}+\frac{1}{2}e_{3}$, and $h_{0}:=e_{1}$, we get
$\left|\left\langle h_{+},h_{0}\right\rangle \right|^{2}=1$, but
\[
\left\langle h_{+},\theta h_{+}\right\rangle \left\Vert h_{0}\right\Vert ^{2}=\left\langle e_{1}+\frac{1}{2}e_{3},e_{1}-\frac{1}{2}e_{3}\right\rangle \left\Vert e_{1}\right\Vert ^{2}=\frac{3}{4}.
\]
Hence condition (\ref{eq:us3b}) does not hold. 

Note that $h_{+}-h_{0}\in\mathscr{H}_{+}^{ex}$, and 
\[
\left\langle h_{+}-h_{0},\theta\left(h_{+}-h_{0}\right)\right\rangle =\left\langle \frac{1}{2}e_{3},-\frac{1}{2}e_{3}\right\rangle =-\frac{1}{4}<0;
\]
i.e., the positivity condition $E_{+}^{ex}\theta E_{+}^{ex}\geq0$
in (\ref{eq:us3a}) is not satisfied. 
\end{example}

\begin{cor}
Let $\mathscr{H}$, $\theta$, and $\mathscr{H}_{0}$, $\mathscr{H}_{\pm}$
be as in \corref{us1}, assume (\ref{eq:us3b}), and let $\mathscr{K}_{\pm}^{ex}$
be the corresponding induced Hilbert spaces; see (\ref{eq:us1}) applied
to $\mathscr{H}_{\pm}^{ex}$. Then the two quotient mappings $\mathscr{H}_{0}\rightarrow\mathscr{K}_{\pm}^{ex}$
are isometric.
\end{cor}

\begin{proof}
Immediate. 
\end{proof}

\subsection{Contractive Inclusions}

As sketched in \figref{sd} below, there are three subspaces naturally
associated with the geometry of a given reflection, $\mathscr{H}_{0}$,
$\mathscr{H}_{+}$, and $\mathscr{H}_{-}$. The last two of these
are determined naturally and directly from the given reflection $\theta$.
The role of the subspace $\mathscr{H}_{0}$ is more subtle, and its
role is more crucial in connection with the Markov property (see \defref{E}
below). Below we specify the possibilities for $\mathscr{H}_{0}$;
see especially the corollary to follow.
\begin{cor}
\label{cor:us2}Let $\mathscr{H}_{+}$, $\mathscr{H}_{0}$, and $\theta$
be as in \corref{us1}, and assume $E_{+}\theta E_{+}\geq0$ (i.e.,
O.S.-positivity). Then TFAE: 
\begin{enumerate}
\item \label{enu:us2a}$\left|\left\langle h_{+},h_{0}\right\rangle \right|^{2}\leq\left\langle h_{+},\theta h_{+}\right\rangle \left\Vert h_{0}\right\Vert ^{2}$,
$\forall h_{0}\in\mathscr{H}_{0}$, $\forall h_{+}\in\mathscr{H}_{+}$
(see (\ref{eq:us3b}));
\item \label{enu:us2b}$\exists l:\mathscr{H}_{0}\rightarrow\mathscr{K}=\left(\mathscr{H}_{+}/\ker\right)^{\sim}$
which is linear bounded and contractive, i.e., 
\begin{equation}
\left\Vert l\left(h_{0}\right)\right\Vert _{\mathscr{K}}\leq\left\Vert h_{0}\right\Vert \;,\forall h_{0}\in\mathscr{H}_{0},\label{eq:bd1}
\end{equation}
(we say that $\mathscr{H}_{0}$ is contractively contained in $\mathscr{K}$),
and (\ref{eq:bd2}) holds. 
\end{enumerate}
\end{cor}

\begin{proof}
(\ref{enu:us2a})$\Longrightarrow$(\ref{enu:us2b}). Assume (\ref{enu:us2a}),
then for $\forall h_{0}\in\mathscr{H}_{0}$ fixed, the map $h_{+}\longmapsto\left\langle h_{0},h_{+}\right\rangle $
is a bounded linear functional, so by Riesz and (\ref{enu:us2a})
$\exists l:\mathscr{H}_{0}\longrightarrow\mathscr{K}=\mathscr{K}^{*}$
(the Hilbert space $\mathscr{K}$ is selfdual) s.t. 
\begin{equation}
\left\langle h_{0},h_{+}\right\rangle =\left\langle l\left(h_{0}\right),q\left(h_{+}\right)\right\rangle _{\mathscr{K}},\;\forall h_{+}\in\mathscr{H}_{+}.\label{eq:bd2}
\end{equation}
The inner product in $\mathscr{H}$ is denoted without subscript,
but the $\mathscr{K}$-inner product is denoted $\left\langle \cdot,\cdot\right\rangle _{\mathscr{K}}$,
so $\left\langle q\left(l_{+}\right),q\left(h_{+}\right)\right\rangle _{\mathscr{K}}=\left\langle l_{+},\theta h_{+}\right\rangle $,
$\forall l_{+}$, $h_{+}\in\mathscr{H}_{+}$. 

By (\ref{enu:us2a}) and (\ref{eq:bd2}), we get $\left\Vert l\left(h_{0}\right)\right\Vert _{\mathscr{K}}\leq\left\Vert h_{0}\right\Vert $,
$\forall h_{0}\in\mathscr{H}_{0}$, which is the assertion in (\ref{enu:us2b}).

(\ref{enu:us2b})$\Longrightarrow$(\ref{enu:us2a}). Assume (\ref{enu:us2b}),
and compute $\left|\left\langle h_{+},h_{0}\right\rangle \right|^{2}$
in (\ref{enu:us2a}). We have 
\begin{eqnarray*}
\left|\left\langle h_{0},h_{+}\right\rangle \right|^{2} & = & \big|\langle\underset{\in\mathscr{K}}{\underbrace{l\left(h_{0}\right)}},q\left(h_{+}\right)\rangle_{\mathscr{K}}\big|^{2}\\
 & \leq & \left\Vert l\left(h_{0}\right)\right\Vert _{\mathscr{K}}^{2}\left\Vert q\left(h_{+}\right)\right\Vert _{\mathscr{K}}^{2}\\
 & \underset{\text{by }\left(\ref{eq:bd1}\right)}{\leq} & \left\Vert h_{0}\right\Vert ^{2}\left\langle h_{+},\theta h_{+}\right\rangle \;\text{which is }\left(\ref{enu:us2a}\right).
\end{eqnarray*}
\end{proof}
\begin{cor}
\label{cor:EP3}Let $\mathscr{H}$, $\mathscr{H}_{\pm}$, $\mathscr{H}_{0}$,
and $\theta$ be as described above, and let $E_{\pm}$, $E_{0}$
be the corresponding projections. Introduce $E_{\pm}^{ex}$ as in
\corref{us1}. Then the following implication holds: 
\begin{align}
E_{+}E_{0}E_{-} & =E_{+}E_{-}\left(\text{the Markov property}\right)\nonumber \\
 & \Downarrow\label{eq:mk7}\\
E_{+}^{ex}E_{0} & =E_{+}^{ex}E_{-}^{ex}.\nonumber 
\end{align}
\end{cor}

\begin{proof}
We have 
\[
E_{+}E_{0}E_{-}=E_{+}E_{-}\Longrightarrow E_{+}^{ex}E_{0}E_{-}^{ex}=E_{+}^{ex}E_{-}^{ex}\Longleftrightarrow E_{+}^{ex}E_{0}=E_{+}^{ex}E_{-}^{ex}
\]
which proves the corollary. We used $E_{0}\leq E_{-}^{ex}$, so $E_{0}E_{-}^{ex}=E_{0}$. 
\end{proof}
\begin{rem}
The purpose of \corref{EP3} is a version (see (\ref{eq:mk7})) of
the Markov property which is closer to the one used for Markov processes;
see \secref{MP}.
\end{rem}

\section{Unitary operators, symmetries, and reflections}

In this section we introduce certain unitary representations which
are given to act on the fixed Hilbert space. So we consider a given
Hilbert space $\mathscr{H}$ which carries a reflection symmetry (in
the sense of Osterwalder-Schrader) as defined in \secref{GR}. If
the unitary representation under consideration, say $U$, is a representation
of a group $G$, then reflection-symmetry will refer to a suitable
semigroup $S$ in $G$, so a sub-semigroup. The setting is of interest
even in the three cases when $G$ is $\mathbb{Z}$, $\mathbb{R}$,
or some Lie group from quantum physics. In the cases $G=\mathbb{Z}$,
or $\mathbb{R}$, the semigroups are obvious, and, in each case, they
define a causality. (The case $G=\mathbb{Z}$ is simply the study
of a single unitary operator.) Nonetheless, the choice of semigroup
in the case when $G$ is a Lie group is more subtle; see \secref{MP}
below. However, many of the important spectral theoretic properties
may be developed initially in the cases $G=\mathbb{Z}$, or $\mathbb{R}$,
where the essential structures are more transparent. 
\begin{lem}
Let $\left\{ U_{t}\right\} _{t\in\mathbb{R}}$ be a unitary one-parameter
group in $\mathscr{H}$, such that $\theta U_{t}\theta=U_{-t}$, $t\in\mathbb{R}$,
and $U_{t}\mathscr{H}_{+}\subset\mathscr{H}_{+}$, $t\in\mathbb{R}_{+}$;
then 
\[
S_{t}=\widetilde{U_{t}}:\mathscr{K}\longrightarrow\mathscr{K},
\]
is a \emph{selfadjoint contraction semigroup}, $t\in\mathbb{R}_{+}$,
i.e., there is a selfadjoint generator $L$ in $\mathscr{K}$ (see
\figref{L}), 
\begin{equation}
\left\langle k,Lk\right\rangle _{\mathscr{K}}\geq0,\;\forall k\in dom\left(L\right),\label{eq:rp9}
\end{equation}
where
\begin{equation}
S_{t}(=\widetilde{U}_{t})=e^{-tL},\;t\in\mathbb{R}_{+},\;\text{and}\label{eq:rp10}
\end{equation}
\begin{equation}
S_{t_{1}}S_{t_{2}}=S_{t_{1}+t_{2}},\;t_{1},t_{2}\in\mathbb{R}_{+}.\label{eq:rp11}
\end{equation}
\end{lem}

\begin{proof}
\begin{flushleft}
See \cite{MR545025,MR887102,MR874059,MR1767902}.
\par\end{flushleft}

\end{proof}
\begin{figure}[H]
\[
\xymatrix{
A\ar[d] & \mathscr{H}\ar[rr]^{U_t=e^{-tA}} & & \mathscr{H} & A^{*}=-A\\
L & \mathscr{K}\ar[rr]^{[S_{t}]_{t\in\mathbb{R}_+}}_{S_{t}=e^{-tL}} & & \mathscr{K} & L^{*}=L,\,\, L\geq0
}
\]

\caption{\label{fig:L}Transformation of skew-adjoint $A$ into selfadjoint
semibounded $L$. }
\end{figure}

\subsection{Two Examples}

We include details below (\exaref{rp}) to stress the distinction
between an abstract Hilbert-norm completion on the one hand, and a
concretely realized Hilbert space on the other.
\begin{example}[\cite{MR1641554,MR1767902}]
\label{exa:rp}Let $0<s<1$ be given, and let $\mathscr{H}=\mathscr{H}_{s}$
be the Hilbert space whose norm $\left\Vert f\right\Vert _{s}$ is
given by
\begin{equation}
\left\Vert f\right\Vert _{s}^{2}=\int_{\mathbb{R}}\int_{\mathbb{R}}\overline{f\left(x\right)}\left|x-y\right|^{s-1}f\left(y\right)dxdy.\label{eqInt.4}
\end{equation}
Let $a\in\mathbb{R}_{+}$ be given, and set
\begin{equation}
\left(U\left(a\right)f\right)\left(x\right)=a^{s+1}f\left(a^{2}x\right).\label{eqInt.5}
\end{equation}
It is clear that then $a\mapsto U\left(a\right)$ is a unitary representation
of the multiplicative group $\mathbb{R}_{+}$ acting on the Hilbert
space $\mathscr{H}_{s}$. It can be checked that $\left\Vert f\right\Vert _{s}$
in (\ref{eqInt.4}) is finite for all $f\in C_{c}\left(\mathbb{R}\right)$
($=$ the space of compactly supported functions on the line). Now
let $\mathscr{H}_{+}$ be the closure of $C_{c}\left(-1,1\right)$
in $\mathscr{H}_{s}$ relative to the norm $\left\Vert \cdot\right\Vert _{s}$
of (\ref{eqInt.4}). It is then immediate that $U\left(a\right)$,
for $a>1$, leaves $\mathscr{H}_{+}$ invariant, i.e., it restricts
to a semigroup of isometries $\left\{ U\left(a\right)\mathrel{;}a>1\right\} $
acting on $\mathscr{K}_{s}$. Setting
\begin{equation}
\left(\theta f\right)\left(x\right)=\left|x\right|^{-s-1}f\left(\frac{1}{x}\right),\quad x\in\mathbb{R}\setminus\left\{ 0\right\} ,\label{eqInt.6}
\end{equation}
we check that $\theta$ is then a period-$2$ unitary in $\mathscr{H}_{s}$,
and that 
\begin{equation}
\theta U\left(a\right)\theta=U\left(a\right)^{\ast}=U\left(a^{-1}\right)\label{eqInt.7}
\end{equation}
and
\begin{equation}
\left\langle f,\theta f\right\rangle _{\mathscr{H}_{s}}\geq0,\quad\forall f\in\mathscr{H}_{+},\label{eqInt.8}
\end{equation}
where $\left\langle \cdot,\cdot\right\rangle _{\mathscr{H}_{s}}$
is the inner product 
\begin{equation}
\left\langle f_{1},f_{2}\right\rangle _{\mathscr{H}_{s}}:=\int_{\mathbb{R}}\int_{\mathbb{R}}\overline{f_{1}\left(x\right)}\left|x-y\right|^{s-1}f_{2}\left(y\right)dxdy.\label{eqInt.9}
\end{equation}
In fact, if $f\in C_{c}\left(-1,1\right)$, the expression in (\ref{eqInt.8})
works out as the following reproducing kernel integral:
\begin{equation}
\int_{-1}^{1}\int_{-1}^{1}\overline{f\left(x\right)}\left(1-xy\right)^{s-1}f\left(y\right)dxdy,\label{eqInt.10}
\end{equation}
and we refer to \cite{Jor86,MR1641554,MR1767902,MR1895530} for more
details on this example. 

Hence up to a constant, the norm $\left\Vert \,\cdot\,\right\Vert _{s}$
of (\ref{eqInt.9}) may be rewritten as
\begin{equation}
\int_{\mathbb{R}}\left|\xi\right|^{-s}\left|\hat{f}\left(\xi\right)\right|^{2}\,d\xi,\label{eqRef.14}
\end{equation}
and the inner product $\left\langle \,\cdot\,,\,\cdot\,\right\rangle _{s}$
as 
\begin{equation}
\int_{\mathbb{R}}\left|\xi\right|^{-s}\overline{\hat{f}_{1}\left(\xi\right)}\hat{f}_{2}\left(\xi\right)\,d\xi,\label{eqRef.15}
\end{equation}
where 
\begin{equation}
\hat{f}\left(\xi\right)=\int_{\mathbb{R}}e^{-i\xi x}f\left(x\right)\,dx\label{eqRef.16}
\end{equation}
is the usual Fourier transform suitably extended to $\mathscr{H}_{s}$,
using Stein's singular integrals. Intuitively, $\mathscr{H}_{s}$
consists of functions on $\mathbb{R}$ which arise as $\left(\frac{d}{dx}\right)^{s}f_{s}$
for some $f_{s}$ in $L^{2}\left(\mathbb{R}\right)$. This also introduces
a degree of ``non-locality'' into the theory, and the functions
in $\mathscr{H}_{s}$ cannot be viewed as locally integrable, although
$\mathscr{H}_{s}$ for each $s$, $0<s<1$, contains $C_{c}\left(\mathbb{R}\right)$
as a dense subspace. In fact, formula (\ref{eqRef.14}), for the norm
in $\mathscr{H}_{s}$, makes precise in which sense elements of $\mathscr{H}_{s}$
are ``fractional'' derivatives of locally integrable functions on
$\mathbb{R}$, and that there are elements of $\mathscr{H}_{s}$ (and
of $\mathscr{K}_{s}$) which are not locally integrable. 

A main conclusion in \cite{MR1895530} for this example is that, when
$\mathscr{H}_{+}$ and $\mathscr{K}$ are as in (\ref{eqInt.10}),
then the natural contractive operator $q$ from (\ref{eq:rp6b})-(\ref{eq:rp7a})
is automatically 1-1, i.e., its kernel is 0.
\end{example}

\begin{rem}
Note that, in general, the spectral type changes in passing from $U$
to $\widetilde{U}$ in \lemref{UK}; see also \figref{rp}. For example,
$U$ from (\ref{eqInt.5}) above has absolutely continuous spectrum,
while $\widetilde{U}$ has purely discrete (atomic) spectrum: When
$a>1$, one checks that the spectrum of $\widetilde{U}\left(a\right)$
is the set $\left\{ a^{-2n}\mathrel{;}n\in\mathbb{N}\right\} $.
\end{rem}

\begin{example}[See \cite{MR1027505}: Reflection Positivity on a Schottky Double]
\label{exa:SD} Let $S$ denote a compact Riemann surface which arises
as a Schottky double of a bordered Riemann surface $T$ with boundary
$\partial T$. A Schottky double $S$ of $T$ is defined as a mirror
image $\overline{T}$, and $S$ is the union of $\overline{T}$, with
$T$ glued on $\partial T$. Thus, the double $S$ of $T$ has an
antiholomorphic involution $\theta T=\overline{T}$, such that $\partial T$
is the set of fixed points of $\theta$. Let $P_{0}\in T$ and define
$P_{\infty}=\theta\left(P_{0}\right)$. The points $P_{0}$ and $P_{\infty}$
then provide reference points on the Riemann surface, which are interchanged
by $\theta$, see \figref{sd}.

The standard case of a real, space-time $S^{1}\times\mathbb{R}$ can
be understood as follows. For $t\in\mathbb{R}$, $x\in S^{1}$, let
$\left(t,x\right)$ denote a space-time point. The map 
\begin{equation}
z=\exp\left[i\left(x+it\right)\right]\label{eq:sd1}
\end{equation}
defines a Riemann sphere $S$, with the half-space $t\geq0$ mapping
into the unit disc $T$ around the origin. Time reflection $\left(t,x\right)\rightarrow\left(-t,x\right)$
in space-time then maps into a reflection $\theta$ on the Riemann
sphere. In local coordinates, 
\begin{equation}
\theta\left(z\right)=\frac{1}{z^{*}}.\label{eq:sd2}
\end{equation}
We identify $\overline{T}$ as the unit disc about infinity, and consider
the Riemann sphere as a Schottky double of the unit disc, with $\theta$
given by (\ref{eq:sd2}). A convenient choice for $P_{0}$ is $P_{0}=0$,
so $P_{\infty}=\infty$, and this comes from Euclidean space-time
points at $t=\mp\infty$. In particular, the operator $\theta$ on
the Riemann sphere can be thought of as a reflection through the unit
circle $\left|z\right|=1$. 

The corresponding ``infinite volume'' space-time $\left(t,x\right)\in\mathbb{R}^{2}$
can also be studied. A compactification is given by the map
\[
z=\frac{x+i\left(t-1\right)}{x+i\left(t+1\right)}.
\]
\end{example}

\begin{figure}
\includegraphics[width=0.4\textwidth]{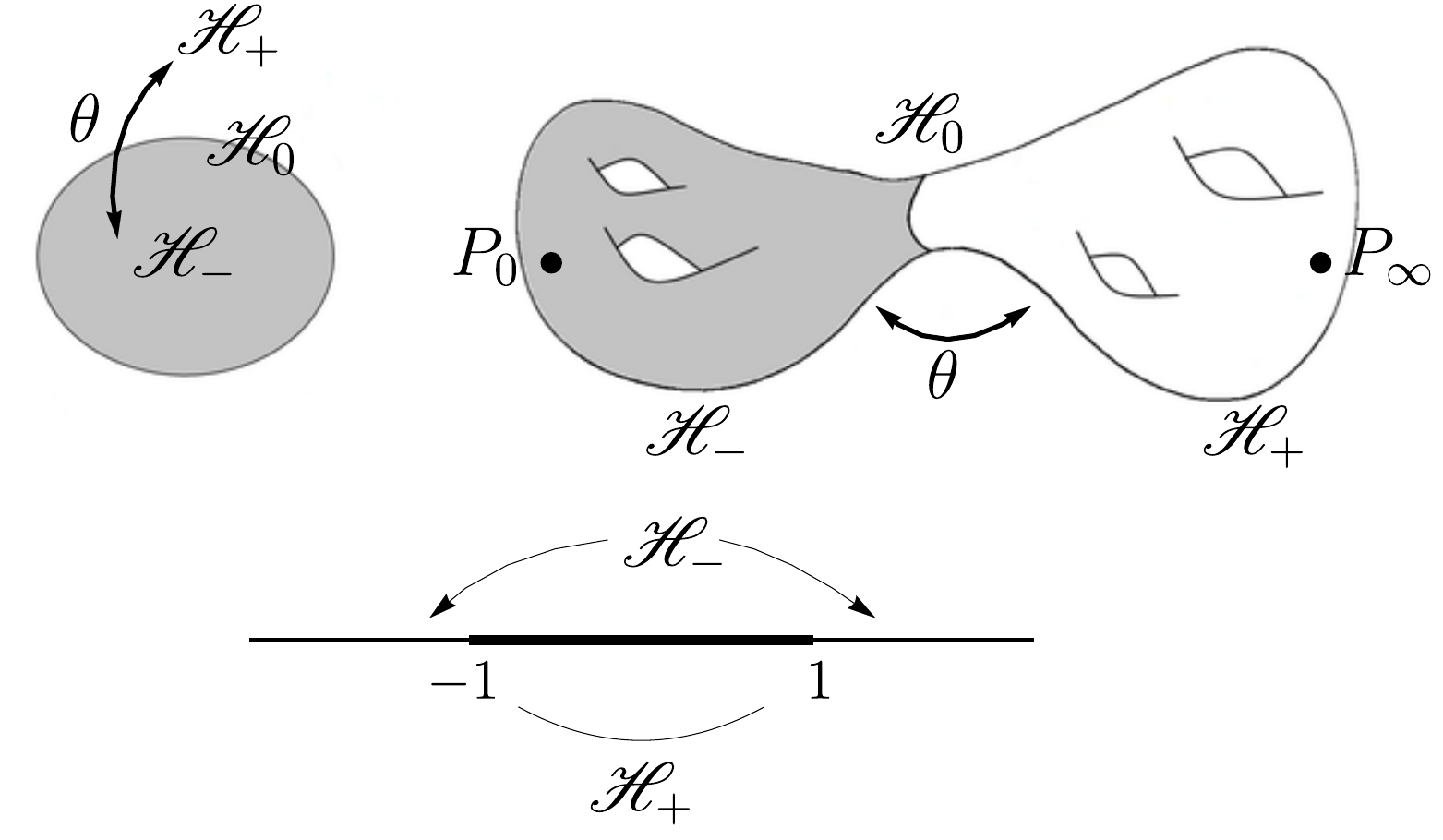}

\caption{\label{fig:sd}(a) The complex plane, inside and outside of the disk;
see sect \ref{sec:ns}, and \cite{MR1895530}. (b) The Schottky double
$S$ of a bordered Riemann surface $T$ with boundary $\partial T$;
see \exaref{SD}. (c) The real line, inside and outside of a fixed
interval; see \exaref{rp}. }
\end{figure}

\section{\label{sec:MO}A characterization of the Markov property: Markov
vs O.-S. positivity}

In the classical case of Gaussian processes (see \cite{AD92,AD93,ABDdS93,AJSV13,AJV14}),
the question of \emph{reflection symmetry} and \emph{reflection positivity}
is of great interest; see, e.g., \cite{MR2862151,MR3051696,JP13,MR3285409,MR3167762,MR3246982,AJV14,MR3392505,MR3513873,MR3614177},
and also \cite{MR0496171,MR0518418,MR659539,MR2337697}. 

Let $\mathscr{H}$ be a given (fixed) Hilbert space; e.g., $\mathscr{H}=L^{2}\left(\Omega,\mathscr{F},\mathbb{P}\right)$,
square integrable random variables, where $\Omega$ is a set (sample
space) with a $\sigma$-algebra of subsets $\mathscr{F}$ (information),
and $\mathbb{P}$ a given probability measure on $\left(\Omega,\mathscr{F}\right)$.
But the question may in fact be formulated for an arbitrary Hilbert
space $\mathscr{H}$, and possible inseparable generally. 

Recall that $\theta:\mathscr{H}\rightarrow\mathscr{H}$ is a \emph{reflection}
if it satisfies $\theta^{*}=\theta$, and $\theta^{2}=I_{\mathscr{H}}$. 
\begin{defn}
\label{def:Ref}Given a Hilbert space $\mathscr{H}$, let 
\begin{equation}
Ref\left(\mathscr{H}\right)=\left\{ \theta:\mathscr{H}\rightarrow\mathscr{H}\mathrel{;}\theta^{*}=\theta,\;\theta^{2}=I_{\mathscr{H}}\right\} ,
\end{equation}
i.e., all reflections in $\mathscr{H}$; see (\ref{eq:rp1}).
\end{defn}

\begin{lem}
\label{lem:theta}Let $\theta:\mathscr{H}\rightarrow\mathscr{H}$
be linear, and let $\mathscr{H}_{\pm}$ be a pair of closed subspaces
of $\mathscr{H}$ with respective projections $E_{\pm}$; then TFAE:
\begin{enumerate}
\item[(i)] $\theta\left(\mathscr{H}_{+}\right)\subseteq\mathscr{H}_{-}$, i.e.,
$\theta$ maps $\mathscr{H}_{+}$ into $\mathscr{H}_{-}$; 
\item[(ii)] $E_{-}\theta E_{+}=\theta E_{+}$. 
\end{enumerate}
\end{lem}

\begin{proof}
Follows from the basic fact that $E_{+}\mathscr{H}=\left\{ h_{+}\in\mathscr{H}\mathrel{;}E_{+}h_{+}=h_{+}\right\} $. 
\end{proof}
We saw in \corref{EP1} that reflection $\theta$ satisfying $E_{+}\theta E_{+}\geq0$
are in 1-1 correspondence with contraction operators $C:\mathscr{H}_{+}\left(\theta\right)\longrightarrow\mathscr{H}_{-}\left(\theta\right)$,
where $\mathscr{H}_{\pm}\left(\theta\right)=\left\{ h_{\pm}\in\mathscr{H}_{\pm}\mathrel{;}\theta h_{\pm}=h_{\pm}\right\} $.
We now fix $E_{+}$, and therefore the subspace $\mathscr{H}_{+}:=E_{+}\mathscr{H}$. 

Let $\theta$, $\theta'$ be a pair of reflections (see \secref{GR}
and \lemref{theta} above), and assume they share the same pair $\mathscr{H}_{\pm}$,
i.e., 
\begin{equation}
\theta\mathscr{H}_{+}\subseteq\mathscr{H}_{-},\quad\theta'\mathscr{H}_{+}\subseteq\mathscr{H}_{-}.\label{eq:tp2}
\end{equation}
\begin{lem}
Let $P$ and $P'$ be the projections onto $\left\{ h\in\mathscr{H}\mathrel{;}\theta h=h\right\} $,
and $\left\{ h'\in\mathscr{H}\mathrel{;}\theta'h'=h'\right\} $, i.e.,
we have $\theta=2P-I_{\mathscr{H}}$, and $\theta'=2P'-I_{\mathscr{H}}$.
Let $C$ and $C'$ be the corresponding contractions: $C:P\mathscr{H}\longrightarrow P^{\perp}\mathscr{H}$,
and $C':P'\mathscr{H}\longrightarrow\left(P'\right)^{\perp}\mathscr{H}$.
Then 
\begin{alignat}{2}
\mathscr{H}_{+} & =Graph\left(C\right) &  & =\left\{ x+Cx\mathrel{;}x\in P\mathscr{H}\right\} \label{eq:tp3}\\
 & =Graph\left(C'\right) &  & =\left\{ x'+C'x'\mathrel{;}x'\in P'\mathscr{H}\right\} ;\nonumber 
\end{alignat}
and 
\begin{align*}
\theta\left(x+Cx\right) & =x-Cx\in\mathscr{H}_{-};\;\text{and}\\
\theta'\left(x'+C'x'\right) & =x'-C'x'\in\mathscr{H}_{-}.
\end{align*}
Moreover, $\left(I_{\mathscr{H}}+C'\right)\big|_{P'\mathscr{H}}$
has a one-sided inverse, and there is an operator 
\begin{equation}
V:P\mathscr{H}\longrightarrow P'\mathscr{H}\label{eq:tp4}
\end{equation}
such that 
\begin{equation}
V\big|_{P\mathscr{H}}=\left(I_{\mathscr{H}}+C'\right)^{-1}\left(I_{\mathscr{H}}+C\right)\big|_{P\mathscr{H}}.\label{eq:tp5}
\end{equation}
\end{lem}

\begin{proof}
This is essentially a consequence of the characterization in \lemref{EP1}
and \corref{EP1}. Indeed, from this, we get the existence of the
operator $V$ as specified in (\ref{eq:tp4}), and satisfying
\begin{equation}
\mathscr{H}_{+}\ni x+Cx=Vx+C'Vx,\label{eq:tp6}
\end{equation}
for all $x\in P\mathscr{H}$. But (\ref{eq:tp6}) may be rewritten
as: 
\[
\left(I_{\mathscr{H}}+C\right)x=\left(I_{\mathscr{H}}+C'\right)Vx;
\]
and the desired conclusion (\ref{eq:tp5}) now follows.
\end{proof}
\begin{defn}
\label{def:E}If $E_{0}$, $E_{\pm}$ are projections in $\mathscr{H}$,
let $\varepsilon=\left(E_{0},E_{\pm}\right)$, and set 
\begin{align}
\mathscr{E}\left(Markov\right) & :=\left\{ \left(E_{0},E_{\pm}\right)\mathrel{;}E_{+}E_{0}E_{-}=E_{+}E_{-}\right\} ,\\
\mathscr{R}\left(\varepsilon\right) & :=\left\{ \theta\in Ref\left(\mathscr{H}\right)\mathrel{;}\theta E_{0}=E_{0},\;\theta E_{+}=E_{-}\theta E_{+},\:\theta E_{-}=E_{+}\theta E_{-}\right\} .\label{eq:e2}
\end{align}

Fix $\theta\in Ref\left(\mathscr{H}\right)$, so that $\theta^{2}=I_{\mathscr{H}}$,
$\theta^{*}=\theta$, set: 
\begin{equation}
\mathscr{E}\left(\theta\right):=\left\{ \left(E_{0},E_{\pm}\right)\mathrel{;}\theta E_{0}=E_{0},\;\theta E_{+}=E_{-}\theta E_{+},\:\theta E_{-}=E_{+}\theta E_{-}\right\} .\label{eq:e3}
\end{equation}
\end{defn}

\begin{rem*}
Recall that $E$ is a projection in $\mathscr{H}$ iff (Def.) $E=E^{2}=E^{*}$;
see \defref{OP}. 

In (\ref{eq:e2}) and (\ref{eq:e3}), the conditions on $\theta$
and the triple of projections $\varepsilon=\left(E_{0},E_{\pm}\right)$
are as follows: $\theta E_{0}=E_{0}$, $\theta\left(\mathscr{H}_{+}\right)\subseteq\mathscr{H}_{-}$
and $\theta\left(\mathscr{H}_{-}\right)\subseteq\mathscr{H}_{+}$;
see \lemref{theta}. 
\end{rem*}
\begin{question*}
(1) Given $\varepsilon$, what is $\mathscr{R}\left(\varepsilon\right)$?
(2) Given $\theta$, what is $\mathscr{E}\left(\theta\right)$? 
\end{question*}
\begin{defn}
Suppose $\varepsilon=\left(E_{0},E_{\pm}\right)$ is given, and $\theta\in\mathscr{R}\left(\varepsilon\right)$.
\begin{enumerate}
\item We say that \emph{reflection positivity} holds iff (Def.)
\begin{equation}
E_{+}\theta E_{+}\geq0,\label{eq:e8}
\end{equation}
also called \emph{Osterwalder-Schrader positivity} (O.S.-p). 
\item Given $\varepsilon$, we say that it satisfies the \emph{Markov property}
iff (Def.)
\begin{equation}
E_{+}E_{0}E_{-}=E_{+}E_{-}.\label{eq:e5}
\end{equation}
\item We set 
\begin{equation}
\mathscr{E}_{OS}\left(\theta\right)=\left\{ \left(E_{0},E_{\pm}\right)\mathrel{;}E_{+}\theta E_{+}\geq0\right\} .
\end{equation}
\end{enumerate}
\end{defn}

\begin{lem}
\label{lem:MP}Suppose (\ref{eq:e5}) holds (the Markov property),
and $\theta\in\mathscr{R}\left(\varepsilon\right)$, then 
\begin{equation}
E_{+}\theta E_{+}\ge0,\label{eq:EOS}
\end{equation}
i.e., the O.S.-positivity condition (\ref{eq:e8}) follows. 
\end{lem}

\begin{proof}
Using the properties in (\ref{eq:e2}), we have
\begin{eqnarray*}
E_{+}\theta E_{+} & \underset{\text{by }\left(\ref{eq:e2}\right)}{=} & E_{+}E_{-}\theta E_{+}\\
 & \underset{\text{by }\left(\ref{eq:e5}\right)}{=} & \left(E_{+}E_{0}E_{-}\right)\theta E_{+}\\
 & = & E_{+}E_{0}\left(E_{-}\theta E_{+}\right)\\
 & \underset{\text{by }\left(\ref{eq:e2}\right)}{=} & E_{+}E_{0}\theta E_{+}\underset{\text{by }\left(\ref{eq:e2}\right)}{=}E_{+}E_{0}E_{+}\geq0,
\end{eqnarray*}
where ``$\ge$'' is in the sense of ordering of selfadjoint operators. 

Note for any pair of projections, we have: $E_{+}E_{0}E_{+}\geq0$,
since 
\[
\left\langle h,E_{+}E_{0}E_{+}h\right\rangle =\left\langle E_{+}h,E_{0}E_{+}h\right\rangle =\left\Vert E_{0}E_{+}h\right\Vert ^{2}\geq0;
\]
where $E_{0}=E_{0}^{*}=E_{0}^{2}$ by definition. 
\end{proof}
Recall the definition of $\mathscr{R}\left(\varepsilon\right)$ and
$\mathscr{R}\left(\varepsilon,U\right)$. \lemref{MP} can be reformulated
as: 
\begin{lem}
\label{lem:MP2}For all $\theta\in\mathscr{R}\left(\varepsilon\right)$,
we have 
\begin{equation}
\mathscr{E}\left(Markov\right)\cap\mathscr{E}\left(\theta\right)\subseteq\mathscr{E}_{OS}\left(\theta\right).
\end{equation}
(See Definitions \ref{def:Ref}, \ref{def:E}, and eq. (\ref{eq:EOS}).) 
\end{lem}

\begin{question*}
Let $\varepsilon=\left(E_{0},E_{\pm}\right)$ be given, and suppose
$E_{+}\theta E_{+}\geq0$, for all $\theta\in\mathscr{R}\left(\varepsilon\right)$,
then does it follow that $E_{+}E_{0}E_{-}=E_{+}E_{-}$ holds? (See
\thmref{OSM} below for an affirmative answer.)
\end{question*}
\begin{thm}
\label{thm:OSM}Given an infinite-dimensional complex Hilbert space
$\mathscr{H}$, let the setting be as above, i.e., reflections, Markov
property, and O.S.-positivity defined as stated. Then 
\begin{equation}
\bigcap_{\theta\in\mathscr{R}\left(\varepsilon\right)}\mathscr{E}_{OS}\left(\theta\right)=\mathscr{E}\left(Markov\right).\label{eq:m1}
\end{equation}
\end{thm}

\begin{rem}
If $\varepsilon$, and $U$ are given as in \secref{GR}, and if $\left(E_{\pm},E_{0},U\right)$
is Markov, then (\ref{eq:m1}) also holds with $\theta$, $U$ satisfying
(\ref{eq:rp2})-(\ref{eq:rp3}). The idea in (\ref{eq:m1}) is that
when a system $\varepsilon$ of projections is fixed as specified
on the RHS in the formula, then on the LHS, we intersect only over
the subset of reflections $\theta$ subordinated to this $\varepsilon$-system.
And similarly when both $\varepsilon$ and $U$ are specified, we
intersect over the smaller set of jointly $\varepsilon$, $U$ subordinated
reflections $\theta$.
\end{rem}

\begin{proof}
We must show that if $\varepsilon:=\left(E_{0},E_{\pm}\right)$ is
given to satisfy the Markov property, i.e., $E_{+}E_{0}E_{-}=E_{+}E_{-}$,
then for all $\theta\in\mathscr{R}\left(\varepsilon\right)$; see
\lemref{EP}. Then by \lemref{MP}, the O.S. property will be automatic.
Now given $\varepsilon$, a reflection $\theta$ may be constructed
via an application of Zorn's lemma to all reflections $\theta$ from
$\mathscr{H}_{+}$ to $\mathscr{H}_{-}$, see (\ref{eq:M4}) below.
Note we can assume that both of the subspaces $\mathscr{H}_{\pm}$
are infinite-dimensional. Hence, to show existence of $\theta\in\mathscr{R}\left(\varepsilon\right)$
as asserted, we must show that, if $\theta$ is initially only defined
on closed subspaces $\mathscr{H}_{+}^{su}\rightarrow\mathscr{H}_{-}^{su}$,
then vectors $h_{\pm}\in\mathscr{H}_{\pm}\ominus\mathscr{H}_{\pm}^{su}$
may be chosen such that $\theta h_{+}=h_{-}$ offers a non-trivial
extension. This is a contradiction since the two subspaces $\mathscr{H}_{\pm}^{su}$
may be chosen maximal by Zorn's lemma. (See also \cite{MR0133686,MR546992}.)

In detail: By \lemref{MP2}, we already have ``$\supseteq$'' in
(\ref{eq:m1}), and we now turn to the other inclusion: 

Given $\left(E_{0},E_{\pm}\right)$, and suppose $\left(E_{0},E_{\pm}\right)\in\mathscr{E}_{OS}\left(\theta\right)$,
$\forall\theta\in\mathscr{R}\left(\varepsilon\right)$. We shall show
that $\left(E_{0},E_{\pm}\right)\in\mathscr{E}\left(Markov\right)$,
i.e., 
\begin{equation}
\bigcap_{\theta\in\mathscr{R}\left(\epsilon\right)}\mathscr{E}_{OS}\left(\theta\right)\subseteq\mathscr{E}\left(Markov\right).\label{eq:M2}
\end{equation}
\emph{Indirect proof of (\ref{eq:M2})}: 

We must prove that if $\left(E_{0},E_{\pm}\right)\notin\mathscr{E}\left(Markov\right)$
then $\exists\theta\in\mathscr{R}\left(\varepsilon\right)$ s.t. $\left(E_{0},E_{\pm}\right)\notin\mathscr{E}_{OS}\left(\theta\right)$. 

Suppose $E_{+}E_{0}E_{-}\neq E_{+}E_{-}$, then $\exists$ $h_{\pm}\neq0$,
$h_{\pm}\in\mathscr{H}_{\pm}$, where $\mathscr{H}_{\pm}:=E_{\pm}\mathscr{H}$,
s.t. 
\begin{equation}
\left\langle h_{+},E_{0}h_{-}\right\rangle \neq\left\langle h_{+},h_{-}\right\rangle ,
\end{equation}
and we may choose these vectors s.t. 
\begin{equation}
\left\langle h_{+},h_{-}\right\rangle \notin[0,\infty).
\end{equation}
See also \thmref{SU}. 

Define $\theta_{0}$ on $\mathbb{C}h_{+}\rightarrow\mathbb{C}h_{-}$
(on 1-dimensional subspaces), $\theta_{0}h_{+}:=h_{-}$; and then
extend it 
\begin{equation}
\theta_{0}\rightarrow\theta:\mathscr{H}_{+}\xrightarrow{\;\theta\;}\mathscr{H}_{-}\:\left(\text{extended}\right),\label{eq:M3}
\end{equation}
to a reflection $\theta$ with initial space $\mathscr{H}_{+}$ and
final space $\mathscr{H}_{-}$, (using again \thmref{SU}) s.t. the
extension $\theta$ satisfies
\[
\theta E_{0}=E_{0},\quad\theta E_{+}=E_{-}\theta E_{+},
\]
i.e., $\left(E_{0},E_{\pm}\right)\in\mathscr{E}\left(\theta\right)$,
see (\ref{eq:M4}). 
\begin{equation}
\begin{matrix}\mathscr{H}_{+}\\
\mathbb{C}h_{+}
\end{matrix}\xymatrix{\text{\boxed{\ensuremath{\begin{matrix}\\
h_{+}\\
\\
\end{matrix}}}}\ar[rrr]_{\theta_{0}}\ar@/^{1.6pc}/[rrr]_{\text{extension}}\ar@/_{1.6pc}/[rrr]_{\theta_{0}\rightarrow\theta} &  &  & \text{\boxed{\ensuremath{\begin{matrix}\\
h_{-}\\
\\
\end{matrix}}}}}
\begin{matrix}\mathscr{H}_{-}\\
\mathbb{C}h_{-}
\end{matrix}\label{eq:M4}
\end{equation}
Then $\left\langle h_{+},\theta h_{+}\right\rangle =\left\langle h_{+},h_{-}\right\rangle \notin[0,\infty)$
by construction, see (\ref{eq:M3})-(\ref{eq:M4}); and so, for this
$\theta$, $\left(E_{0},E_{\pm}\right)\notin\mathscr{E}_{OS}\left(\theta\right)$,
and (\ref{eq:M2}) is proved.
\end{proof}
\begin{example}[Markov property]
\label{exa:MP}Let $\mathscr{H}=L^{2}\left(\Omega,\mathscr{F},\mathbb{P}\right)$,
where 
\begin{itemize}
\item $\Omega$: sample space;
\item $\mathscr{F}$: total information;
\item $\mathscr{F}_{-}$: information from the past (or inside);
\item $\mathscr{F}_{+}$: information from the future (predictions), or
from the outside; 
\item $\mathscr{F}_{0}$: information at the present.
\end{itemize}
Let $\mathbb{E}\left(\cdot\mid\mathscr{F}_{0}\right)$, $\mathbb{E}\left(\cdot\mid\mathscr{F}_{\pm}\right)$
be the corresponding conditional expectations, and the Markov property
(\ref{eq:e5}) then takes the form $\mathbb{E}_{0}\mathscr{H}=\mathscr{H}_{0}$,
$\mathbb{E}_{\pm}\mathscr{H}=\mathscr{H}_{\pm}$. 

The Markov process is a probability system: 
\begin{equation}
\mathbb{E}\left(\mathbb{E}\left(\psi_{+}\mid\mathscr{F}_{-}\right)\mid\mathscr{F}_{0}\right)=\mathbb{E}\left(\psi_{+}\mid\mathscr{F}_{-}\right),\label{eq:mk1}
\end{equation}
for $\forall\psi_{+}$ (random variables conditioned by $\mathscr{F}_{+}$
= the future); or, if $\mathscr{F}_{0}\subseteq\mathscr{F}_{-}$,
it simplifies to: 
\begin{equation}
\mathbb{E}\left(\psi_{+}\mid\mathscr{F}_{-}\right)=\mathbb{E}\left(\psi_{+}\mid\mathscr{F}_{0}\right),\;\forall\psi_{+}\in\mathscr{H}_{+}.\label{eq:mk2}
\end{equation}
For more details on this point, see \secref{MP} below.
\end{example}

\begin{question*}
Do we have analogies of O.S.-positivity (see (\ref{eq:e8})) in the
free probability setting? That is, in the setting of free probability
and non-commuting random variables. 
\end{question*}

\section{\label{sec:ns}A model for reflection symmetry}

This is a following up on a result in \cite{MR1895530}, and we offer
the following analysis as reflection operators $\theta$: Let $\mathscr{H}_{i}$,
$i=1,2$, be a pair of Hilbert spaces; we shall assume that they are
both separable and infinite dimensional. Set $\mathscr{H}=\left(\begin{smallmatrix}\underset{\bigoplus}{\mathscr{H}_{1}}\\
\mathscr{H}_{2}
\end{smallmatrix}\right)$, i.e., column vectors. The Hilbert norm in $\mathscr{H}$ is the
usual one:
\begin{equation}
\left\Vert \begin{pmatrix}h_{1}\\
h_{2}
\end{pmatrix}\right\Vert _{\mathscr{H}}^{2}=\left\Vert h_{1}\right\Vert _{\mathscr{H}_{1}}^{2}+\left\Vert h_{2}\right\Vert _{\mathscr{H}_{2}}^{2}.
\end{equation}
Set 
\begin{equation}
\theta=\begin{pmatrix}1 & 0\\
0 & -1
\end{pmatrix},\label{eq:m2}
\end{equation}
more precisely, $\theta\left(h_{1}\oplus h_{2}\right)=h_{1}\oplus\left(-h_{2}\right)$. 
\begin{thm}
Let $\mathscr{H}_{i}$, $i=1,2$, and $\theta$ be as above. A system
$E_{0}$, $E_{\pm}$ with subspaces $\mathscr{H}_{0}$, $\mathscr{H}_{\pm}$
in $\mathscr{H}$, satisfies the O.S.-condition $\left\langle h_{+},\theta h_{+}\right\rangle _{\mathscr{H}}\geq0$,
$\forall h_{+}\in\mathscr{H}_{+}$, if and only if there is a \emph{contractive}
linear operator $C:\mathscr{H}_{1}\rightarrow\mathscr{H}_{2}$ such
that $\mathscr{H}_{1}=Graph\left(C\right)$, $\mathscr{H}_{-}=Graph\left(-C\right)$,
and  
\begin{equation}
\mathscr{H}_{0}=\begin{pmatrix}\underset{\bigoplus}{\mathscr{H}_{1}}\\
\mathbf{0}
\end{pmatrix}=\left\{ \begin{pmatrix}h_{1}\\
0
\end{pmatrix}\mathrel{;}h_{1}\in\mathscr{H}_{1}\right\} .
\end{equation}
\end{thm}

\begin{proof}
We refer to \cite{MR1895530} for details, but the easy implication
is as follows: Given $C:\mathscr{H}_{1}\rightarrow\mathscr{H}_{2}$,
and $\theta$ be as in (\ref{eq:m2}), then 
\[
\left\langle \begin{pmatrix}h_{1}\\
Ch_{1}
\end{pmatrix},\theta\begin{pmatrix}h_{1}\\
Ch_{1}
\end{pmatrix}\right\rangle _{\mathscr{H}}=\left\Vert h_{1}\right\Vert _{\mathscr{H}_{1}}^{2}-\left\Vert Ch_{1}\right\Vert _{\mathscr{H}_{2}}^{2}\geq0
\]
iff $C$ is contractive. One checks that the converse implication
holds as well. 
\end{proof}
\begin{thm}
Let $\mathscr{H}=\left(\begin{smallmatrix}\underset{\bigoplus}{\mathscr{H}_{1}}\\
\mathscr{H}_{2}
\end{smallmatrix}\right)$ be as above, and let $C:\mathscr{H}_{1}\rightarrow\mathscr{H}_{2}$
be a contraction. Set 
\begin{alignat}{2}
\mathscr{H}_{+} & =Graph\left(C\right), &  & \mathscr{H}_{-}=Graph\left(-C\right),\label{eq:m4}\\
\mathscr{H}_{0} & =\begin{pmatrix}\underset{\bigoplus}{\mathscr{H}_{1}}\\
\mathbf{0}
\end{pmatrix},\;\text{and} & \quad & E_{0}=\begin{pmatrix}1 & 0\\
0 & 0
\end{pmatrix}.\label{eq:m5}
\end{alignat}
Then the Markov property 
\begin{equation}
E_{+}E_{0}E_{-}=E_{+}E_{-}\label{eq:m6}
\end{equation}
holds if and only if $C=0$. 
\end{thm}

\begin{proof}
Let $E_{\pm}$ be the projections corresponding to the two subspaces
$\mathscr{H}_{\pm}$ in (\ref{eq:m4}). One checks that
\begin{equation}
E_{+}=\begin{pmatrix}\left(1+C^{*}C\right)^{-1} & \left(1+C^{*}C\right)^{-1}C^{*}\\
C\left(1+C^{*}C\right)^{-1} & 1-\left(1+CC^{*}\right)^{-1}
\end{pmatrix},\label{eq:m7}
\end{equation}
and 
\begin{equation}
E_{-}=\begin{pmatrix}\left(1+C^{*}C\right)^{-1} & -\left(1+C^{*}C\right)^{-1}C^{*}\\
-C\left(1+C^{*}C\right)^{-1} & 1-\left(1+CC^{*}\right)^{-1}
\end{pmatrix}.\label{eq:m8}
\end{equation}
In abbreviated form we have 
\[
E_{+}=\begin{pmatrix}P_{11} & P_{12}\\
P_{21} & P_{22}
\end{pmatrix}
\]
with the operator entries as specified in (\ref{eq:m7}), and so 
\[
E_{-}=\begin{pmatrix}P_{11} & -P_{12}\\
-P_{21} & P_{22}
\end{pmatrix}.
\]
A further computation yields
\[
E_{+}E_{0}E_{-}=\begin{pmatrix}P_{11}^{2} & -P_{11}P_{12}\\
P_{21}P_{11} & -P_{21}P_{12}
\end{pmatrix},
\]
and 
\[
E_{+}E_{-}=\begin{pmatrix}P_{11}^{2}-P_{12}P_{21} & -P_{11}P_{12}+P_{12}P_{22}\\
P_{21}P_{11}-P_{22}P_{21} & -P_{21}P_{12}+P_{22}^{2}
\end{pmatrix}.
\]

Hence the Markov property (\ref{eq:m6}) holds iff $P_{12}P_{21}=0$.
Note that $P_{21}=P_{12}^{*}$. Using the operator entries from (\ref{eq:m7})-(\ref{eq:m8}),
we conclude that (\ref{eq:m6}) holds iff $C=0$, in which case $E_{+}=E_{-}=E_{0}$,
where $E_{0}$ is as in (\ref{eq:m5}). 
\end{proof}
\begin{rem}
The matrix $E_{+}$ (i.e., the characteristic matrix of $C$) from
(\ref{eq:m7}) is obtained as follows: 

Let $\begin{pmatrix}x\\
y
\end{pmatrix}\in\mathscr{H}$, then 

\begin{gather}
\begin{pmatrix}P_{11} & P_{12}\\
P_{21} & P_{22}
\end{pmatrix}\begin{pmatrix}x\\
y
\end{pmatrix}=\begin{pmatrix}P_{11}x+P_{12}y\\
P_{21}x+P_{22}y
\end{pmatrix}\in Graph\left(C\right),\nonumber \\
\Updownarrow\nonumber \\
C:\left(P_{11}x+P_{12}y\right)\longmapsto P_{21}x+P_{22}y,\nonumber \\
\Updownarrow\nonumber \\
CP_{11}=P_{21},\quad CP_{12}=P_{22}.\label{eq:EP1}
\end{gather}
On the other hand, $E_{+}^{\perp}=I-E_{+}$ is the projection from
$\mathscr{H}$ onto $V\left(Graph\left(C^{*}\right)\right)$, where
$V:=\begin{pmatrix}0 & -1\\
1 & 0
\end{pmatrix}$. It follows that, 
\begin{gather}
C^{*}:P_{21}x-(1-P_{22})y\longmapsto(1-P_{11})x-P_{12}y,\nonumber \\
\Updownarrow\nonumber \\
C^{*}P_{21}=1-P_{11},\quad C^{*}(1-P_{22})=P_{12}.\label{eq:EP2}
\end{gather}
Solving (\ref{eq:EP1})-(\ref{eq:EP2}), we get $E_{+}$ as in (\ref{eq:m7}).
\end{rem}

\section{\label{sec:MP}Markov processes and Markov reflection positivity}

In the above, we considered systems $\mathscr{H}$, $E_{0}$, $E_{\pm}$,
$\theta$, and $U$, where $\mathscr{H}$ is a fixed Hilbert space;
$E_{0}$, $E_{\pm}$ are then three given projections in $\mathscr{H}$,
$\theta$ is a reflection, and $U$ is a unitary representation of
a Lie group $G$. 

The axioms for the system are as follows:
\begin{enumerate}
\item $\theta E_{0}=E_{+}$;
\item $E_{+}\theta E_{-}=\theta E_{-}$;
\item $E_{-}\theta E_{+}=\theta E_{+}$;
\item the O.S.-positivity holds, i.e., 
\begin{equation}
E_{+}\theta E_{+}\ge0;\label{eq:mp1}
\end{equation}
\item $\theta U\theta=U^{*}$, or $\theta U\left(g\right)\theta=U(g^{-1})$.
\end{enumerate}
It is further assumed that, for some sub-semigroup $S\subset G$,
we have $U\left(s\right)\mathscr{H}_{+}\subset\mathscr{H}_{+}$, $\forall s\in S$;
or equivalently, 
\begin{equation}
E_{+}U\left(s\right)E_{+}=U\left(s\right)E_{+},\;s\in S.
\end{equation}
From \secref{ns}, it is clear that the additional Markov-restriction
\begin{equation}
E_{+}E_{0}E_{-}=E_{+}E_{-}\label{eq:mp3}
\end{equation}
is ``very'' strong. Moreover, if $\theta$ is fixed, we saw that
(\ref{eq:mp3}) $\Longrightarrow$ (\ref{eq:mp1}) (see \lemref{MP}).

Here we note that (\ref{eq:mp3}) holds in a natural setting of path
space analysis: 

\subsection{\label{subsec:PS}Probability Spaces}

By a probability space we mean a triple $\left(\Omega,\mathscr{F},\mathbb{P}\right)$,
where $\Omega$ is a set (the sample space), $\mathscr{F}$ is a $\sigma$-algebra
of subsets (information), and $\mathbb{P}$ is a probability measure
defined on $\mathscr{F}$. Measurable functions $\psi$ on $\left(\Omega,\mathscr{F}\right)$
are called \emph{random variables}. If $\psi$ is a random variable
in $L^{2}\left(\Omega,\mathscr{F},\mathbb{P}\right)$, we say that
it has finite second moment. An indexed family of random variables
is called a \emph{stochastic process}, or a \emph{random field}.

Let $\left(\Omega,\mathscr{F},\mathbb{P}\right)$ be a fixed probability
space. The expectation will be denoted 
\begin{equation}
\mathbb{E}\left(\psi\right)=\int_{\Omega}\psi\,d\mathbb{P},\label{eq:mp4}
\end{equation}
if $\psi$ is a given random variable on $\left(\Omega,\mathscr{F},\mathbb{P}\right)$.

We shall be primarily interested in the $L^{2}\left(\Omega,\mathscr{F},\mathbb{P}\right)$
setting. 

If $\psi$ is a random variable (or a random field) then 
\begin{equation}
\psi^{-1}\left(\mathscr{B}\right)\subseteq\mathscr{F},\label{eq:rv1}
\end{equation}
where $\mathscr{B}$ is the Borel $\sigma$-algebra of subsets of
$\mathbb{R}$. 

For every sub-$\sigma$-algebra $\mathscr{G}\subset\mathscr{F}$,
there is a unique conditional expectation
\begin{equation}
\mathbb{E}\left(\cdot\mid\mathscr{G}\right):L^{2}\left(\Omega,\mathscr{F},\mathbb{P}\right)\longrightarrow L^{2}\left(\Omega,\mathscr{F},\mathbb{P}\right).\label{eq:rv2}
\end{equation}
In fact $\mathscr{G}$ defines a closed subspace in $L^{2}\left(\Omega,\mathscr{F},\mathbb{P}\right)$,
the closed span of the indicator functions $\left\{ \chi_{S}\mathrel{;}S\in\mathscr{G}\right\} $,
and $\mathbb{E}\left(\cdot\mid\mathscr{G}\right)$ in (\ref{eq:rv2})
will then be the projection onto this subspace. 

If $\mathscr{G}\subset\mathscr{F}$ is as in (\ref{eq:rv1}) then,
for random variables $\psi_{1}\in L^{2}\left(\mathscr{G},\mathbb{P}\right)$,
and $\psi_{2}\in L^{2}\left(\mathscr{F},\mathbb{P}\right)$, we have
\begin{equation}
\mathbb{E}\left(\psi_{1}\psi_{2}\right)=\mathbb{E}\left(\psi_{1}\mathbb{E}\left(\psi_{2}\mid\mathscr{G}\right)\right).
\end{equation}
If $\psi_{1}$ is also in $L^{\infty}\left(\mathscr{G},\mathbb{P}\right)$,
then 
\begin{equation}
\mathbb{E}\left(\psi_{1}\psi_{2}\mid\mathscr{G}\right)=\psi_{1}\mathbb{E}\left(\psi_{2}\mid\mathscr{G}\right).
\end{equation}

The following property is immediate from this: If $\mathscr{G}_{i}$,
$i=1,2$, are two sub-$\sigma$-algebras with $\mathscr{G}_{1}\subseteq\mathscr{G}_{2}$,
then for all $\psi\in L^{2}\left(\Omega,\mathscr{F},\mathbb{P}\right)$
we have 
\begin{equation}
\mathbb{E}\left(\mathbb{E}\left(\psi\mid\mathscr{G}_{2}\right)\mid\mathscr{G}_{1}\right)=\mathbb{E}\left(\psi\mid\mathscr{G}_{1}\right).
\end{equation}
Indeed, this is immediate from the equivalences in \defref{OP}.

Let $\left\{ \psi_{t}\right\} _{t\in\mathbb{R}}$ be a random process
in the given probability space $\left(\Omega,\mathscr{F},\mathbb{P}\right)$.
For $t\in\mathbb{R}$, set $\mathscr{F}_{t}:=$ the $\sigma$-algebra
$\left(\subseteq\mathscr{F}\right)$ generated by the random variables
$\left\{ \psi_{s}\mathrel{;}s\leq t\right\} $. When $t$ is fixed,
we set $\mathscr{B}_{t}:=$ the $\sigma$-algebra generated by the
random variable $\psi_{t}$. We say that $\left\{ \psi_{t}\right\} _{t\in\mathbb{R}}$
is a \emph{Markov-process} iff (Def.), for every $t>s$, and every
measurable function $f$, we have 
\begin{equation}
\mathbb{E}\left(f\circ\psi_{t}\mid\mathscr{F}_{s}\right)=\mathbb{E}\left(f\circ\psi_{t}\mid\mathscr{B}_{s}\right)\label{eq:pm1}
\end{equation}
where $\mathbb{E}\left(\cdot\mid\mathscr{F}_{s}\right)$, and $\mathbb{E}\left(\cdot\mid\mathscr{B}_{s}\right)$,
refer to the corresponding conditional expectations. It is well known
that the Markov property is equivalent to the following \emph{semigroup
property}:

Set 
\begin{equation}
\left(S_{t}f\right)\left(x\right):=\mathbb{E}\left(f\circ\psi_{t}\mid\psi_{0}=x\right),\label{eq:pm2}
\end{equation}
then, for all $t,s\geq0$, we have 
\begin{equation}
S_{t+s}=S_{t}S_{s}.\label{eq:pm3}
\end{equation}
So the semigroup law (\ref{eq:pm3}) holds if and only if the Markov
property (\ref{eq:pm1}) holds.

In order to make a direct comparison with the present Markov property
from \corref{EP3}, it is convenient to restrict attention to stationary
processes; and we now turn to the details of that below. 

\subsection{The covariance operator}

Now let $V$ be a real vector space; and assume that it is also a
LCTVS, locally convex topological vector space. Let $G$ be a Lie
group, $U$ a unitary representation of $G$; and let $\left\{ \psi_{v,g}\right\} _{\left(v,g\right)\in V\times G}$
be a real valued stochastic process s.t. $\psi_{v,g}\in\mathscr{H}=L^{2}\left(\Omega,\mathscr{F},\mathbb{P}\right)$,
and 
\begin{equation}
\mathbb{E}\left(\psi_{v,g}\right)=0,\;\left(v,g\right)\in V\times G.
\end{equation}
We further assume that a reflection $\theta$ is given, and that 
\begin{equation}
\theta\left(\psi_{v,g}\right)=\psi_{v,g^{-1}},\;\left(v,g\right)\in V\times G.
\end{equation}

Let $\left(v_{i},g_{i}\right)$, $i=1,2$, be given, and set 
\begin{equation}
\mathbb{E}\left(\psi_{v_{1},g_{1}}\psi_{v_{2},g_{2}}\right)=\left\langle v_{1},r\left(g_{1},g_{2}\right)v_{2}\right\rangle \label{eq:mp7}
\end{equation}
where $\left\langle \cdot,\cdot\right\rangle $ is a fixed positive
definite Hermitian inner product on $V$. Hence (\ref{eq:mp7}) determines
a function $r$ on $G\times G$; it is operator valued, taking values
in operators in $V$. This function is called the \emph{covariance
operator}.

To sketch the setting for the Markov property (\ref{eq:mp3}), we
shall make two specializations (these may be removed!):
\begin{enumerate}
\item[(i)] $G=\mathbb{R}$, $S=\mathbb{R}_{+}\cup\left\{ 0\right\} =[0,\infty)$,
and 
\item[(ii)] the process is stationary; i.e., referring to (\ref{eq:mp7}) we
assume that the covariance operator $r$ is as follows: 
\begin{equation}
\mathbb{E}\left(\psi_{v_{1},t_{1}}\psi_{v_{2},t_{2}}\right)=\left\langle v_{1},r\left(t_{1}-t_{2}\right)v_{2}\right\rangle ,\label{eq:mp8}
\end{equation}
$\forall t_{1},t_{2}\in\mathbb{R}$, $\forall v_{1},v_{2}\in V$. 
\end{enumerate}
In this case, the O.S.-condition (\ref{eq:mp1}) is considered for
the following three sub-$\sigma$-algebras $\mathscr{A}_{0}$, $\mathscr{A}_{\pm}$
in $\mathscr{F}$: 
\begin{enumerate}[label=]
\item $\mathscr{A}_{0}=$ the $\sigma$-algebra generated by $\left\{ \psi_{v,0}\right\} _{v\in V}$, 
\item $\mathscr{A}_{+}=$ the $\sigma$-algebra generated by $\left\{ \psi_{v,t}\right\} _{v\in V,t\in[0,\infty)}$,
and
\item $\mathscr{A}_{-}=$ the $\sigma$-algebra generated by $\left\{ \psi_{v,t}\right\} _{v\in V,t\in(-\infty,0]}$.
\end{enumerate}
The corresponding conditional expectations will be denoted as follows:
\begin{equation}
\begin{split}\mathbb{E}_{0}\left(\psi\right) & =\mathbb{E}\left(\psi\mid\mathscr{A}_{0}\right),\;\text{and}\\
\mathbb{E}_{\pm}\left(\psi\right) & =\mathbb{E}\left(\psi\mid\mathscr{A}_{\pm}\right).
\end{split}
\label{eq:mp9}
\end{equation}
The corresponding closed subspaces in $\mathscr{H}=L^{2}\left(\Omega,\mathscr{F},\mathbb{P}\right)$
will be denoted $\mathscr{H}_{0}$, $\mathscr{H}_{\pm}$, respectively,
and we shall consider the positivity conditions (\ref{eq:mp1}) O.S.-p,
and (\ref{eq:mp3}) Markov, in this context. 

Translating a theorem in \cite{MR0496171}, we arrive at the following: 
\begin{thm}[A. Klein \cite{MR0496171}]
Let the stationary stochastic process $\left\{ \psi_{v,t}\right\} $,
$\left(v,t\right)\in V\times\mathbb{R}$, be as specified above, and
let $\left\{ r\left(t\right)\right\} _{t\in\mathbb{R}}$ be the covariance
operator. Set $\theta\left(\psi_{v,t}\right):=\psi_{v,-t}$, $t\in\mathbb{R}$.
Assume $\left\langle \psi_{+},\theta\psi_{+}\right\rangle \geq0$,
$\forall\psi_{+}\in\mathscr{H}_{+}$, then for $\forall n\in\mathbb{N}$,
$\forall\left\{ v_{i}\right\} _{i=1}^{n}\subset V$, $\forall\left\{ t_{i}\right\} _{i=1}^{n}\subset\mathbb{R}_{+}\cup\left\{ 0\right\} $,
we have 
\begin{equation}
\sum_{i}\sum_{j}\left\langle v_{i},r\left(t_{i}+t_{j}\right)v_{j}\right\rangle \geq0;\label{eq:mp10}
\end{equation}
which is the O.S.-positivity condition. 

Moreover, the Markov property $E_{+}E_{0}E_{-}=E_{+}E_{-}$ holds
iff $r\left(\cdot\right)$ is a semigroup, i.e., 
\begin{equation}
r\left(t+s\right)=r\left(t\right)r\left(s\right),\label{eq:mp11}
\end{equation}
for $\forall s,t\in[0,\infty)$. 

In particular, in the case of stationary processes, when O.S.-positivity
is assumed, then two conditions hold:
\begin{enumerate}
\item the covariance function $r\left(\cdot\right)$ is positive definite:
\[
\sum_{i}\sum_{j}\left\langle v_{i},r\left(t_{i}-t_{j}\right)v_{j}\right\rangle \geq0;\;\text{and}
\]
\item condition (\ref{eq:mp10}) holds as well. 
\end{enumerate}
\end{thm}

\begin{rem}
In the scalar case, a list of stationary positive definite, and Gaussian
O.S.-positive, covariance functions $\left\{ r\left(t\right)\right\} _{t\in\mathbb{R}}$
includes: 
\begin{itemize}[itemsep=5pt]
\item $e^{-a\left|t\right|}$, $a>0$, fixed; 
\item ${\displaystyle \frac{1-e^{-b\left|t\right|}}{b\left|t\right|}}$,
$b>0$, fixed; 
\item ${\displaystyle \frac{1}{1+\left|t\right|}}$; 
\item ${\displaystyle \frac{1}{\sqrt{1+\left|t\right|}}e^{-\frac{\left|t\right|}{1+\left|t\right|}}}$,
$t\in\mathbb{R}$. 
\end{itemize}
But of these, only the first one $r\left(t\right):=e^{-a\left|t\right|}$
is also the generator of a Markov system; it is the Ornstein-Uhlenbeck
process. The corresponding semigroup is called the Ornstein-Uhlenbeck
semigroup and it is of independent interest in applications to stochastic
analysis (L\'evy processes) and to mathematical physics; see e.g.,
\cite{MR0343815,MR0343816,MR0436832,MR0518418,MR887102,MR3385977,MR3601243,MR3392505,MR3584630},
and also \cite{MR0496171,MR0518418,MR545025,MR832989,MR887102,MR1641554}. 
\end{rem}

\begin{rem}
As outlined in recent papers by the first named author with Neeb and
Olafsson (\cite{MR1641554,MR1767902,MR3513873}), the extension of
the results also holds in the context of Lie groups $G$, with semigroups
$S\subseteq G$. The above deals with the case $G=\mathbb{R}$, $S=[0,\infty)$. 
\end{rem}

\begin{cor}
Let $\left\{ \psi_{v,t}\right\} $ be as specified above, $\theta\,\psi_{v,t}=\psi_{v,-t}$,
and assume Osterwalder-Schrader positivity holds. Let $\mathscr{K}$
denote the Hilbert completion of $span\left\{ \psi_{v,t}\mathrel{;}t\geq0\right\} $
with respect to the induced inner product from (\ref{eq:mp10}). Then
a selfadjoint and contractive semigroup $\left\{ R\left(s\right)\mathrel{;}s\geq0\right\} $
is well defined by $R\left(s\right)\psi_{v,t}:=\psi_{v,t+s}$; i.e.,
$\left\{ R\left(s\right)\right\} _{s\ge0}$ is a selfadjoint contractive
semigroup of operators in $\mathscr{K}$, $R\left(s+s'\right)=R\left(s\right)R\left(s'\right)$. 
\end{cor}

\begin{proof}
Immediate. Note that 
\begin{align*}
\left\langle R\left(s\right)\psi_{v_{1},t_{1}},\psi_{v_{2},t_{2}}\right\rangle _{\mathscr{K}} & =\left\langle \psi_{v_{1},t_{1}},R\left(s\right)\psi_{v_{2},t_{2}}\right\rangle _{\mathscr{K}}\\
 & =\left\langle v_{1},r\left(t_{1}+t_{2}+s\right)v_{2}\right\rangle _{V}.
\end{align*}
\end{proof}
\begin{acknowledgement*}
The co-authors thank the following colleagues for helpful and enlightening
discussions: Professors Daniel Alpay, Sergii Bezuglyi, Ilwoo Cho,
A. Jaffe, Paul Muhly, K.-H. Neeb, G. Olafsson, Wayne Polyzou, Myung-Sin
Song, and members in the Math Physics seminar at The University of
Iowa. 

\bibliographystyle{amsalpha}
\bibliography{ref}
\end{acknowledgement*}

\end{document}